\documentclass[12p]{amsart}
\usepackage{amssymb}
\usepackage{amsmath}
\usepackage{amsfonts}
\usepackage{geometry}
\usepackage{graphicx}
\usepackage{mathrsfs,amssymb}
\usepackage{xcolor}
\usepackage{mathtools}

\usepackage{hyperref}
\usepackage{cleveref}

\theoremstyle{plain}

\newtheorem{thm}{THEOREM}[section]

\newtheorem{conjecture}[thm]{Conjecture}
\newtheorem{corollary}[thm]{Corollary}
\newtheorem{definition}[thm]{Definition}

\newtheorem{lemma}[thm]{Lemma}
\newtheorem{proposition}[thm]{Proposition}
\newtheorem{remark}[thm]{Remark}
\newtheorem{theorem}[thm]{Theorem}
\numberwithin{equation}{section}

\newcommand{\ls}{\lesssim}

\newcommand{\lmd}{\lambda}
\newcommand{\mb}{\mathbb}
\newcommand{\be}{\begin{equation}}
\newcommand{\ee}{ \end{equation}}
\newcommand{\vn}[1]{\|#1\|}
\newcommand{\vm}[1]{\left|#1\right|}
\newcommand{\lng}{\langle}
\newcommand{\rng}{\rangle}
\newcommand{\lpp}{\left[}
\newcommand{\rpp}{\right]}
\newcommand*\jb[1]{\left\langle #1 \right\rangle} 
\newcommand{\defeq}{\vcentcolon=}
\newcommand{\ep}{\varepsilon}
\newcommand{\pt}{\partial}
\newcommand{\dd}{\,\mathrm{d}}

\usepackage{xcolor}

\begin{document}

\title[Instability of the soliton for gKdV]{Instability of the soliton for the focusing, mass-critical generalized KdV equation}

\author{Benjamin Dodson}
\address{Department of Mathematics, Johns Hopkins University, Baltimore, MD, 21218}
\email{bdodson4@jhu.edu}

\author{Cristian Gavrus}
\address{Department of Mathematics, Johns Hopkins University, Baltimore, MD, 21218}
\email{cgavrus1@jhu.edu} 

\begin{abstract}
In this paper we prove instability of the soliton for the focusing, mass-critical generalized KdV equation. We prove that the solution to the generalized KdV equation for any initial data with mass smaller than the mass of the soliton and close to the soliton in $L^{2}$ norm must eventually move away from the soliton.
\end{abstract}

\maketitle

\tableofcontents

\section{Introduction}

In this paper we prove $L^{2}$ instability of the soliton for the focusing, mass-critical, generalized KdV equation
\begin{equation}\label{1.1}
u_{t} = -(u_{xx} + u^{5})_{x}, \qquad u(0,x) = u_{0} \in L^{2}(\mathbb{R}).
\end{equation}
This equation is called mass-critical because the scaling leaving \eqref{1.1} invariant, i.e.
$$
u(t, x) \mapsto \lambda^{\frac{1}{2}} u\left(\lambda^{3} t, \lambda x\right)
$$
leaves the $L^{2}$ norm, or mass, invariant. The mass of a solution, defined by
$$
M(u(t)):=\int_{\mathbb{R}}|u(t, x)|^{2} d x
$$
is conserved.

Recently, \cite{dodson2017global} proved that the defocusing, mass-critical generalized KdV equation
\begin{equation}\label{1.2}
u_{t} = -(u_{xx} - u^{5})_{x}, \qquad u(0,x) = u_{0} \in L^{2}(\mathbb{R}),
\end{equation}
is globally well-posed and scattering for any $u_{0} \in L^{2}(\mathbb{R})$.  The proof of the defocusing result used the concentration compactness method. Namely, a result of \cite{killip2009mass} combined with a scattering result of \cite{defocusing2016global} for the defocusing nonlinear Schr{\"o}dinger equation,
\begin{equation}\label{1.3}
i u_{t} + u_{xx} = |u|^{4} u, \qquad u(0,x) = u_{0} \in L^{2}(\mathbb{R}),
\end{equation}
implies that for scattering to fail for $(\ref{1.2})$, there must exist a nonzero, almost periodic solution to $(\ref{1.2})$.

\begin{definition}[Almost periodic solution]\label{d1.1}
Suppose $u$ is a strong solution to $(\ref{1.1})$ on the maximal interval of existence $I$. Such a solution $u$ is said to be almost periodic (modulo symmetries) if there exist continuous functions $ N(t) : I \rightarrow (0, \infty)$ and $x(t) : I \rightarrow \mathbb{R}$, such that
\begin{equation}\label{1.4}
\{ v(t,x) = N(t)^{-1/2} u(t, N(t)^{-1} x + x(t)) : t \in I\}
\end{equation}
is contained in a compact subset of $L^{2}(\mathbb{R})$. See also section \ref{sec:almost:per} for an equivalent condition. 
\end{definition}

Then \cite{defocusing2016global} proved that in the defocusing case, there does not exist a nonzero, almost periodic solution to $(\ref{1.2})$, which implies scattering for the defocusing equation $(\ref{1.2})$. The proof used an interaction Morawetz estimate based upon the argument in \cite{Tao2006TwoRO}, which proved there does not exist a soliton for the defocusing, generalized KdV equation. \medskip

For the focusing generalized KdV equation, there exists the soliton $u(t,x) = Q(x - t)$, where
\begin{equation}\label{1.6}
Q(x) = \frac{3^{1/4}}{\cosh^{1/2}(2x)} > 0.
\end{equation}
The function $Q(x)$ solves the elliptic equation
\begin{equation}\label{1.7}
Q_{xx} + Q^{5} = Q,
\end{equation}
so therefore, $Q(x - t)$ solves $(\ref{1.1})$. Note that $Q(x - t)$ is an almost periodic solution to $(\ref{1.1})$. Meanwhile, for the focusing, mass-critical nonlinear Schr{\"o}dinger equation,
\begin{equation}\label{1.5}
i u_{t} + u_{xx} = -|u|^{4} u, \qquad u(0,x) = u_{0} \in L^{2}(\mathbb{R}),
\end{equation}
$u(t,x) = e^{it} Q(x)$ gives a soliton solution.\medskip

The paper \cite{dodson2016global} proved that the focusing nonlinear Schr{\"o}dinger equation $(\ref{1.5})$ is scattering for initial data below the ground state, $\| u_{0} \|_{L^{2}} < \| Q \|_{L^{2}}$. It is conjectured that the same is also true for the focusing, generalized KdV equation.
\begin{conjecture}\label{c1.2}
If $\| u_{0} \|_{L^{2}} < \| Q \|_{L^{2}}$, then the solution to $(\ref{1.1})$ is globally well-posed and scattering.
\end{conjecture}
It can be verified that if Conjecture $\ref{c1.2}$ is true, then this implies that there does not exist an almost periodic solution to $(\ref{1.1})$ below the ground state.
\begin{conjecture}\label{c1.3}
There does not exist a nonzero, almost periodic solution $u$ to $(\ref{1.1})$ satisfying $0 < \| u \|_{L^{2}} < \| Q \|_{L^{2}}$.
\end{conjecture}
However, unlike in the defocusing case, Conjecture $\ref{c1.3}$ does not imply Conjecture $\ref{c1.2}$. This is because \cite{killip2009mass} states that if $(\ref{1.5})$ is globally well-posed and scattering when $\| u \|_{L^{2}} < \| Q \|_{L^{2}}$, Conjecture $\ref{c1.3}$ implies Conjecture $\ref{c1.2}$ when $0 < \| u \|_{L^{2}} < \sqrt{\frac{5}{6}} \| Q \|_{L^{2}}$. In the defocusing case, the presence of the constant $\sqrt{\frac{5}{6}}$ is unimportant, because scattering for the defocusing nonlinear Schr{\"o}dinger equation holds for any finite mass. However, in the focusing case, the constant $\sqrt{\frac{5}{6}}$ becomes quite important, since it is conjectured that $(\ref{1.1})$ scatters for any $\| u_{0} \|_{L^{2}} < \| Q \|_{L^{2}}$.\medskip

Conjecture $\ref{c1.2}$ would also imply instability of the soliton in an $L^{2}$-sense. For any initial data $u_{0} \in L^{2}$, $\| u_{0} \|_{L^{2}} < \| Q \|_{L^{2}}$, the solution to $(\ref{1.1})$ would scatter to a free solution, and thus the solution would approach distance
$$(\| Q \|_{L^{2}}^{2} + \| u_{0} \|_{L^{2}}^{2})^{1/2}$$ from any translation or rescaling of the soliton as $t \rightarrow \pm \infty$.\medskip

In a remarkable series of works, \cite{merle2001existence}, \cite{martel2000liouville}, \cite{martel2001instability}, \cite{martel2002blow}, \cite{merle1998blow}, \cite{martel2002stability} proved, among many nice results, the instability of the soliton in an $H^{1}$ sense, for initial data with mass greater than or equal to the soliton. In fact, they proved something more, that there initial data arbitrarily close to the soliton in $H^{1}$-norm, which eventually move away from the soliton in an $L^{2}$-sense. See \cite{martel2014blow} and \cite{MR3372073} for results in a weighted $L^{2}$ space.\medskip

In this paper we show that there are no almost periodic solutions to $(\ref{1.1})$ which are uniformly close to $Q(x)$ in $ L^2_x $ modulo symmetries. 

\begin{definition}
If a maximal-lifespan strong solution $ u $ to \eqref{1.1} on $ I $ satisfies
\begin{equation}\label{1.8}
\sup_{t \in I} \inf_{\lambda_{0}, x_{0}} \| u(t,x) - \frac{1}{\lambda_{0}^{1/2}} Q(\frac{x - x_{0}}{\lambda_{0}}) \|_{L^{2}(\mathbb{R})} \leq \delta
\end{equation}
then we say $ u $ is $ \delta $-close to $ Q $.  It is readily seen that the infimum is attained and the values $ \lambda_{0}(t), x_{0}(t)$ which attain the minimum can be chosen to be continuous. 

%
\end{definition}

The main result is
\begin{theorem}\label{main.thm}
There exists $\delta > 0$ sufficiently small such that there does not exist a maximal-lifespan solution to $(\ref{1.1})$ with $\| u_{0} \|_{L^{2}} < \| Q \|_{L^{2}}$ satisfying \eqref{1.8}.
\end{theorem}

In other words, Theorem \ref{main.thm} states that there no solutions $ \delta $-close to $ Q $. A consequence of this fact is that for any initial data satisfying $\| u_{0} \|_{L^{2}} < \| Q \|_{L^{2}}$, the solution to $(\ref{1.1})$ with such initial data must eventually move a distance $\delta > 0$ away from the soliton, modulo translations and rescalings, where $\delta > 0$ is a small, fixed constant.\medskip

We split Theorem \ref{main.thm} into two statements. The first part reduces the study to the existence of almost-periodic solutions. 

\begin{theorem} \label{almost_periodic_prop}
Suppose $ u:I \times \mb{R} \to \mb{R} $ is a maximal-lifespan strong solution with $ \vn{u_0}_{L^2} < \vn{Q}_{L^2} $ to the mass-critical focusing gKdV equation \eqref{1.1} which is $ \delta $-close to $ Q $. Then, if $ \delta $ is small enough, there exists an almost periodic modulo symmetries maximal-lifespan (strong) solution $ v $ which is $ \delta $-close to $ Q $ with mass less than $ Q $.
\end{theorem}

The proof is given in Section \ref{Sec.reduction.to.almost.periodic} and it relies essentially on a Palais-Smale result based on the Airy linear profile decomposition, decoupling and an approximation of gKdV solutions by NLS solutions, which are tools developed in \cite{shao2009}, \cite{Tao2006TwoRO}, \cite{killip2009mass} and reviewed in Section \ref{Sec.prelim}. See \cite{10.1093/imrn/rny164} for a similar argument in the case of the mass-critical nonlinear Schr{\"o}dinger equation.\medskip

Once we have this reduction, we prove that such solutions cannot exist.

\begin{theorem}\label{t1.1}
There are no almost periodic solutions to \eqref{1.1} with mass less than $ Q $ which are $ \delta $-close to $ Q $, if $ \delta $ is small enough. 
\end{theorem}

The proof of Theorem $\ref{t1.1}$ combines the ideas of \cite{dodson2017global} and in \cite{merle2001existence}, \cite{martel2000liouville}, \cite{martel2001instability}, \cite{martel2002blow}, \cite{merle1998blow}, \cite{martel2002stability}, \cite{martel2014blow}, and \cite{MR3372073}. The proof of scattering in \cite{dodson2017global} reduced an almost periodic solution to three scenarios: a self-similar solution, a double rapid cascade solution, and a quasisoliton solution. The arguments used in excluding the self-similar and double rapid cascade solutions can also be used to exclude an almost periodic solution to $(\ref{1.1})$ with mass less than the soliton, regardless of whether it is close to the soliton or not. However, in the defocusing case, the interaction Morawetz estimate developed in \cite{dodson2017global} used \cite{Tao2006TwoRO}, and there is no analog to \cite{Tao2006TwoRO}, even for a solution with mass below the mass of the soliton. Instead, we rely on the Morawetz arguments in \cite{merle2001existence}, \cite{martel2000liouville}, \cite{martel2001instability}, \cite{martel2002blow}, \cite{merle1998blow}, \cite{martel2002stability}, \cite{martel2014blow}, and \cite{MR3372073}. These Morawetz estimates depend very much on the fact that the solution is close to the soliton in an $L^{2}$-sense, and can be used to show that a solution cannot stay close to the soliton for the entire time of its existence.\medskip

\textbf{Acknowledgements:} The authors are grateful to Jonas L{\"u}hrmann, Yvann Martel, and Daniel Tataru for several helpful conversations concerning this problem. The first author also acknowleges the support of NSF grant DMS-$1764358$.

\section{Preliminaries} \label{Sec.prelim}

\subsection{Notation and linear estimates}

We will write $ x \ls y $ to denote $ x \leq C y$ for a uniform constant $ C >0 $. We denote $ \jb{x} = (1+x^2)^{1/2}$.
The one-dimensional Fourier transform is defined by 
$$
\hat{f}(\xi):=\frac{1}{(2 \pi)^{1 / 2}} \int_{\mathbb{R}} e^{-i x \xi} f(x) \dd x, \qquad \xi \in \mb{R}
$$
which is used to define the linear propagator and fractional differentiation operators by 
$$
\widehat{e^{-t \partial_{x}^{3}} f} (\xi)= e^{-it \xi^3} \hat{f} (\xi), \qquad \widehat{\left|\partial_{x}\right|^{s} f}(\xi):=|\xi|^{s} \hat{f}(\xi).
$$

For an interval $ I $ one considers the mixed norms on $ I \times \mathbb{R} $
\begin{align*}
\vn{F}_{L_{t}^{p} L_{x}^{q}(I \times \mathbb{R})} = \Big(\int_{I}\big(\int_{\mathbb{R}}|F(t, x)|^{q} \dd x \big)^{p / q} \dd t \Big)^{1 / p}, \\ 
\vn{F}_{L_{x}^{p} L_{t}^{q}(I \times \mathbb{R}) } = \Big(\int_{\mathbb{R}}\big(\int_{I}|F(t,x )|^{q} \dd t\big)^{p / q} \dd x\Big)^{1 / p},    
\end{align*}
with the standard modification when $ p = \infty $ or $ q = \infty $.
We recall the dispersive estimate
$$
\left\|e^{-t \partial_{x}^{3}} u_{0}\right\|_{L_{x}^{p}(\mathbb{R})} \lesssim t^{-\frac{2}{3}\left(\frac{1}{2}-\frac{1}{p}\right)}\left\|u_{0}\right\|_{L_{x}^{p^{\prime}}(\mathbb{R})}, \qquad 2 \leq p \leq \infty.
$$

We will consider weakly convergent sequences in $ L_{x}^{2}(\mathbb{R}) $, i.e. $ f_n \rightharpoonup f $ if 
$$ \lng f_n,  g \rng =  \int_{\mathbb{R}} f_n(x) \bar{g}(x) \dd x \to  \int_{\mathbb{R}} f(x) \bar{g}(x) \dd x \qquad \forall \  g \in  L_{x}^{2}(\mathbb{R}).
$$
By approximation arguments, it sufficies to check this condition for all $ g \in C_c(\mathbb{R}) $.
A basic fact which we will be using tacitly is that if  $ f_n \rightharpoonup f $ then $$ \vn{f}_{L_{x}^{2}(\mathbb{R})} \leq 
\liminf_{n \to \infty} \vn{f_n}_{L_{x}^{2}(\mathbb{R})}. $$
\subsection{Solutions to gKdV}

Throughout this paper we will consider strong solutions, defined as follows. 

\begin{definition} \label{deff} \ 

\begin{enumerate}
\item $A$ function $u: I \times \mathbb{R} \rightarrow \mathbb{R}$ on a non-empty interval $0 \in I \subset \mathbb{R}$ is $a$ (strong) solution to \eqref{1.1} if it lies in the class $C_{t}^{0} L_{x}^{2}(J \times \mathbb{R}) \cap L_{x}^{5} L_{t}^{10}(J \times \mathbb{R})$ for any compact $J \subset I$ and obeys the Duhamel formula
$$
u(t)=e^{-t \partial_{x}^{3}} u_{0}-\int_{0}^{t} e^{-(t-\tau) \partial_{x}^{3}} \partial_{x}\left(u^{5}(\tau)\right) d \tau.
$$
We say that $u$ is a maximal-lifespan solution if the solution cannot be extended to any strictly larger interval. We say that u is a global solution if $I=\mathbb{R}$.
\item The scattering size is defined to be 
$$S_{I}(u)=\int_{\mathbb{R}}\left(\int_{I}|u(t, x)|^{10} d t\right)^{1 / 2} d x=\|u\|_{L_{x}^{5} L_{t}^{10}(I \times \mathbb{R})}^{5}. $$
\item We say that a solution $u$ to \eqref{1.1}  blows up forward in time if there exists $t_{1} \in I$ such that $
S_{\left[t_{1}, \sup (I)\right)}(u)=\infty
$
and that $u$ blows up backward in time if there exists a time $t_{1} \in I$ such that
$
S_{\left(\inf (I), t_{1}\right]}(u)=\infty.
$
\item We say that $u$ scatters forward/backwards in time if there exists a unique $u_{\pm} \in L_{x}^{2}(\mathbb{R})$ such that
\be \label{scattering}
\lim _{t \rightarrow \pm \infty}\left\|u(t)-e^{-t \partial_{x}^{3}} u_{\pm}\right\|_{L_{x}^{2}(\mathbb{R})}=0.
\ee
\item The symmetry group $ G $ is defined as the set of unitary transformations
$$
G=\{ g_{x_{0}, \lambda}: L_{x}^{2}(\mathbb{R}) \rightarrow L_{x}^{2}(\mathbb{R}) | \ (x_0,\lmd) \in \mb{R} \times (0,\infty), \ g_{x_{0}, \lambda} f (x):=\lambda^{-\frac{1}{2}} f\left(\lambda^{-1}\left(x-x_{0}\right)\right)
\}. 
$$
For $u: I \times \mathbb{R} \rightarrow \mathbb{R}$, one defines $T_{g_{x_{0}, \lambda}} u: \lambda^{3} I \times \mathbb{R} \rightarrow \mathbb{R}$  by
$$
T_{g_{x_{0}, \lambda}} u (t, x):=\lambda^{-\frac{1}{2}} u\left(\lambda^{-3} t, \lambda^{-1}\left(x-x_{0}\right)\right).
$$
\end{enumerate}
\end{definition}

$T_{g} u$ solves \eqref{1.1} with initial data $g u_{0}$ if $ u $ is a solution. Moreover,  scattering sizes are invariant
$$ S_{\lmd^3 I}( T_{g} u )  =   S_{I}(u) , \qquad g \in G. $$

We note that $ G $ is a Lie group and the map $ g \mapsto T_g $ is a homomorphism. Giving the operators in $ G $ the strong operator topology, then the identification $ (x_0,\lmd) \mapsto g_{x_{0}, \lambda} $ is a homeomorphism between $ \mb{R} \times (0,\infty) $ and $ G $. 
Thus we say $ g_{x_n, \lmd_n} \to \infty $ if $ \vm{x_n} + \lmd_n + \lmd_n^{-1} \to \infty $. Moreover, in that case $ g_{x_n, \lmd_n}$ converges to $ 0 $ in the weak operator topology.

\

The $ L^2 $ local well-posedness theory of \eqref{1.1} was established by Kenig, Ponce, Vega in \cite{KenigPonceVega}.

\begin{theorem}[Local well-posedness \cite{KenigPonceVega}] For any $u_{0} \in L_{x}^{2}(\mathbb{R})$ and $t_{0} \in \mathbb{R},$ there exists a unique solution $ u $ to \eqref{1.1} with $u\left(t_{0}\right)=u_{0}$ which has maximal lifespan. Let  $ I $ denote the lifespan of $ u $. Then:
\begin{enumerate}
\item I is an open neighborhood of $t_{0}$.
\item If $\sup (I) / \inf (I)$ is finite then $u$ blows up forward / backward in time. 
\item If sup $(I)=+\infty$ and $u$ does not blow up forward in time, then u scatters forward in time. Conversely, given $u_{+} \in L_{x}^{2}(\mathbb{R})$ there is a unique solution to \eqref{1.1} in a neighborhood of $\infty$ so that \eqref{scattering} holds. One can define scattering backward in time in a completely analogous manner.
\item If $M\left(u_{0}\right)$ is sufficiently small then $u$ is a global solution which does not blow up either forward or backward in time and
$
S_{\mathbb{R}}(u) \lesssim M(u)^{5 / 2}.
$
\item Uniformly continuous dependence on initial data holds, see Corollary \ref{sol.conv}.
\end{enumerate}
\end{theorem}

\subsection{Stability and corollaries} 
The stability theory of the generalized KdV equation \eqref{1.1} is discussed in detail in \cite{killip2009mass}. 

\begin{lemma}[Short time stability \cite{killip2009mass}(Lemma 3.3)] \label{short.time.stability}
Let $ I $ be an interval with $ 0 \in I $. Suppose $ \tilde{u} : I \times \mb{R} \to \mb{R} $ is a solution to 
\begin{align} \label{approx.kdv}
(\partial_t + \partial_x^3) \tilde{u} + \partial_x(\tilde{u}^5) &=e, \\
 \tilde{u}(0,x) &=\tilde{u}_0(x), \nonumber
\end{align}
for some function $ e $ such that 
$$
\vn{\tilde{u}}_{L^{\infty}_t L^2_x(I \times \mb{R})} \leq M, 
$$
for some $ M > 0 $. Let $ u_0 $ be such that 
$$
\vn{u_0 - \tilde{u}_0}_{L^2_x} \leq M',
$$ 
for some $ M' \geq 0 $. Assume the smallness conditions
\begin{align}
\vn{\tilde{u}}_{L^5_x L^{10}_t(I \times \mb{R})   } & \leq \ep_0, \\
\vn{ e^{-t \partial_x^3} (u_0 - \tilde{u}_0) }_{L^5_x L^{10}_t(I \times \mb{R}) } & \leq \ep, \\
\vn{\vm{\pt_x}^{-1} e }_{L^1_x L^{2}_t(I \times \mb{R}) } & \leq \ep,
\end{align}
for some small $ 0< \ep< \ep_0=\ep_0(M,M') $. Then there exists a solution $ u:I \times \mb{R} \to \mb{R} $ to \eqref{1.1} with initial data $ u(0)=u_0 $ satisfying 
\begin{align}
\vn{u-\tilde{u}}_{ L^5_x L^{10}_t(I \times \mb{R})  } + \vn{u^5-\tilde{u}^5}_{L^1_x L^{2}_t(I \times \mb{R})} & \ls \ep ,\\
 \vn{u-\tilde{u}}_{L^{\infty}_t L^2_x(I \times \mb{R})} +\vn{\vm{\pt_x}^{1/6} ( u-\tilde{u})}_{ L^6_{t,x} (I \times \mb{R})}  & \ls M'+\ep. 
\end{align}
\end{lemma}

Iterating this lemma over small intervals also a long-time stability result can be obtained, see \cite[Theorem 3.1]{killip2009mass}. To keep track of the number of small intervals one uses the following bound. 

\begin{lemma} \label{num.intervals}
Let $ v \in L^5_x L^{10}_t(J \times \mb{R}) $ for an interval $ J $. Divide $ J $ into $ N $ intervals $ [t_k,t_{k+1}] $ such that $ \vn{v}_{ L^5_x L^{10}_t([t_k,t_{k+1}] \times \mb{R})} \simeq \ep_0 $ for every $ 1 \leq k \leq N-1  $ and $ \vn{v}_{ L^5_x L^{10}_t([t_N,t_{N+1}] \times \mb{R})} \ls \ep_0 $, for a fixed $ \ep_0 > 0$. Then the number of intervals $ N $ is finite and $ N \ls (1+ \vn{v}_{ L^5_x L^{10}_t ( J \times \mb{R})} / \ep_0 )^{10} $. 
\end{lemma}
\begin{proof}
See \cite[Thm 3.1 -first part]{killip2009mass}
\end{proof}

As a consequence, one has

\begin{corollary}[Uniformly continuous dependence on initial data] \label{sol.conv}
Consider solutions $ v \in L^5_x L^{10}_t(J \times \mb{R}) $ to \eqref{1.1}. For every $ \ep > 0 $ there exists $ \delta = \delta(\ep, \vn{v(0)}_{L^2_x} ,  \vn{v}_{ L^5_x L^{10}_t ( J \times \mb{R})} )$ such that if $ \vn{u_0-v(0)}_{L^2_x} \leq \delta $, then there exists a solution to \eqref{1.1} defined on $ J $, with initial data $ u(0)=u_0 $ such that 
$$
\vn{u-v}_{ L^{\infty}_t L^2_x \cap L^5_x L^{10}_t(J \times \mb{R}) } \leq \ep. 
$$
\end{corollary}

Finally, we can use stability to prove a compactness property for the the transformations associated to solutions that are $ \delta$-close to $ Q $. 

\begin{lemma} \label{group.compactness}
There exists $ \delta>0 $ such that the following statement holds. Let $ u:I \times \mb{R} \to \mb{R} $ be a strong solution to \eqref{1.1} such that
\be \label{epp.close}
\vn{ g(t) u(t)- Q}_{L^2_x} \leq \delta, \qquad \forall \ t \in I, 
\ee
with $ g(0)=g_{0,1} $ (the identity), $ 0 \in I $, $ g(t) \in G $. Then for any $ t \in I $ there exists a compact set $ K_t  $ depending only on $ \vm{t}, \ \vn{u}_{ L^5_x L^{10}_t([0,t] \times \mb{R})} $ and $ M(u) $ such that $ g(t) \in K_t $.  
\end{lemma}
\begin{proof} Without loss of generality suppose $ t>0 $ is fixed. We split $ [0,t] $ into $ N $ intervals as in Lemma \ref{num.intervals} where $ \ep_0 $ is the constant in Lemma \ref{short.time.stability}. Then $ N  $ depends on  $ M(u) $ and $ \vn{u}_{ L^5_x L^{10}_t([0,t] \times \mb{R})}. $

We do an induction argument.  We prove that if the statement holds for $ t_k $ then it also holds for $ s \in [t_k, t_{k+1} ] $. At $ t=0 $ we have $ K_0 = \{ g_{0,1} \}. $

From $ \vn{ u(t_k)- g(t_k)^{-1}  Q}_{L^2_x} \leq \delta $ using Lemma \ref{short.time.stability} we deduce
$$ 
\vn{u(s) - g(t_k)^{-1} Q(\cdot - \frac{s-t_k}{\lmd_k^3} ) }_{L^2} \ls \delta, 
$$
From this and \eqref{epp.close} at $ s $ we find
$$
\vn{Q - g(s) g(t_k)^{-1} g_{(s-t_k)/\lmd_k^3,1} Q}_{L^2} \ls \delta. 
$$
For $ \delta $ small enough, $  g(s) g(t_k)^{-1} g_{(s-t_k)/\lmd_k^3,1}  $ will lie in a small compact neighborhood of the identity parametrized by $ (x,\lmd) \in [-\eta,\eta] \times [r,R] $ for some $ \eta > 0 $ and $ 0<r < 1 <R $. Therefore $ g(s) $ has to be in a compact set. Moreover, denoting $ g(t_k) = g_{x_k, \lmd_k} $, one checks inductively that
$$
\lmd_k \in [r^{k-1},R^{k-1} ], \qquad \vm{x_k} \leq (1+R)^{k-2} \eta + \frac{R^{k-1}}{r^{3 (k-2)}} t_k. 
$$
which implies the stated dependence.
\end{proof}

\subsection{Almost periodicity} \label{sec:almost:per}

As a consequence of the Arzela-Ascoli theorem we know that precompactness of a family of functions in $L_{x}^{2}(\mathbb{R})$ is equivalent to it being bounded in  $L_{x}^{2}(\mathbb{R})$ and the existence of a function $C(\eta) $ so that 
$$
\int_{|x| \geq C(\eta)}|f(x)|^{2} \dd x+\int_{|\xi| \geq C(\eta)}|\hat{f}(\xi)|^{2} \dd \xi \leq \eta \qquad \forall \ \eta >0,
$$
holds for all the functions. Therefore, the almost periodicity condition \eqref{1.4} is equivalent to 
$$
\int_{\vm{x-x(t)} \geq C(\eta)/N(t)} \vm{u(t,x)}^2 \dd x  + \int_{|\xi| \geq C(\eta) N(t)}|\hat{u}(t,\xi)|^{2} \dd \xi \leq \eta \qquad \forall \ \eta >0.
$$

\subsection{The embedding of NLS into gKdV} 

We now review the approximation of solutions to gKdV by certain modulated, rescaled versions of solutions to NLS discussed in \cite{Tao2006TwoRO}, \cite{Christ2003AsymptoticsFM}, \cite{killip2009mass}. 

We cite the following theorem from \cite[Thm. 4.1]{killip2009mass}, which was initially conditional on the global well-posedness and scattering of the focusing NLS below the ground state, which was subsequently proved in \cite{DODSON20151589}.  We will only need this theorem for small data (in which case the existence part is automatic), and specifically we will use the approximations \eqref{smoothapprox}, \eqref{unifapproxu}. Here 
\be \label{unt}
\tilde{u}_{n}^T(t, x) \defeq \left\{\begin{array}{ll}
\operatorname{Re}\left[ e^{i x \xi_{n} \lmd_n +i t (\xi_{n} \lmd_n)^{3}}  V_{n}\left(3 \xi_{n} \lmd_n t, x+3 (\xi_{n} \lmd_n)^{2} t\right)\right], & \text { when }|t| \leq \frac{T}{3 \xi_{n} \lmd_n} \\
\exp \left\{-\left(t- \frac{T}{3 \xi_{n} \lmd_n}\right) \partial_{x}^{3}\right\} \tilde{u}_{n}\left(\frac{T}{3 \xi_{n} \lmd_n}\right), & \text { when } t>\frac{T}{3 \xi_{n} \lmd_n} \\
\exp \left\{-\left(t+ \frac{T}{3 \xi_{n} \lmd_n}\right) \partial_{x}^{3}\right\} \tilde{u}_{n}\left(- \frac{T}{3 \xi_{n} \lmd_n}\right), & \text { when } t<-\frac{T}{3 \xi_{n} \lmd_n}
\end{array}\right.
\ee
is defined in terms of certain frequency-localized solutions $ V_n $ and $ V $ to NLS such that 
\be \label{VnV}
\vn{V_n- V}_{L_{t}^{\infty} L_{x}^{2}(\mathbb{R} \times \mathbb{R})} \to 0. 
\ee

\begin{theorem}[Oscillatory profiles \cite{killip2009mass}] \label{oscilatory.profiles}
Let $\phi \in L_{x}^{2}$ with $M(\phi)<2 \sqrt{\frac{6}{5}} M(Q)$. Let $  (\lmd_{n})_{n \geq 1}, (\xi_{n})_{n \geq 1} \subset(0, \infty)$,  with $\xi_{n} \lmd_n \rightarrow \infty$ and let $ (t_{n} )_{n \geq 1} \subset \mathbb{R}$ such that $3 \xi_{n} \lmd_n t_{n}$ converges to some
$T_{0} \in[-\infty, \infty] .$ Then, for $n$ sufficiently large there exists a global solution $ \tilde{v}_{n}$ to \eqref{1.1} with initial data at time $t=t_{n}$ given by
$$
\tilde{v}_{n} \left(t_{n}, x\right)=e^{-t_{n} \partial_{x}^{3}} \operatorname{Re}[ e^{i x \xi_{n} \lmd_n } \phi(x)]
$$
The solution obeys the global spacetime bounds
$$
\vn{ \left|\partial_{x}\right|^{1 / 6} \tilde{v}_{n} }_{L_{t, x}^{6}(\mathbb{R} \times \mathbb{R})}+\left\|\tilde{v}_{n} \right\|_{L_{x}^{5} L_{t}^{10}(\mathbb{R} \times \mathbb{R})} \lesssim_{\phi} 1
$$
and for every $\varepsilon>0$ there exist $n_{\varepsilon} \in \mathbb{N}$ and $\psi_{\varepsilon} \in C_{c}^{\infty}(\mathbb{R} \times \mathbb{R})$ so that, for all $n \geq n_{\varepsilon}$ one has
\be \label{smoothapprox}
\vn{ \tilde{v}_{n} (t, x)-\operatorname{Re} [e^{i x \xi_{n} \lmd_n +i t (\xi_{n} \lmd_n)^{3}} \psi_{\varepsilon} \left(3 \xi_{n} \lmd_n t, x+3 (\xi_{n} \lmd_n)^{2} t\right) ] }_{L_{x}^{5} L_{t}^{10}(\mathbb{R} \times \mathbb{R})} \leq \varepsilon.
\ee
Moreover, defining $ \tilde{u}_{n}^T $ by \eqref{unt}, one has the approximation 
\be \label{unifapproxu}
\lim_{T \to \infty} \lim _{n \rightarrow \infty} \left\|  \tilde{v}_{n}-\tilde{u}_{n}^T \right\|_{L_{t}^{\infty} L_{x}^{2}(\mathbb{R} \times \mathbb{R})} = 0 
\ee
\end{theorem}
We note that \eqref{unifapproxu} is obtained in the proof of  \cite[Thm. 4.1]{killip2009mass}. 

\subsection{The Airy profile decomposition and decoupling}

\begin{definition} \label{asym.orth}
\begin{enumerate}
 \item We say that two sequences $ (\Gamma_n^1)_{n \geq 1} =(\lmd_n^1, \xi_n^1, x_n^1, t_n^1 )_{n \geq 1} $ and $ (\Gamma_n^2)_{n \geq 1} =(\lmd_n^2, \xi_n^2, x_n^2, t_n^2 )_{n \geq 1} $ in $ (0,\infty) \times \mb{R}^3 $ are
asymptotically orthogonal  if
\begin{multline*}
\frac{\lmd_n^1}{\lmd_n^2} + \frac{\lmd_n^2}{\lmd_n^1} + \sqrt{ \lmd_n^1 \lmd_n^2} \vm{\xi_n^1 - \xi_n^2} + \jb{ \lmd_n^1 \xi_n^1 \lmd_n^2 \xi_n^2}^{\frac{1}{2}} 
\vm{\frac{ (\lmd_n^1)^3 t^1_n- (\lmd_n^2)^3 t^2_n }{ (\lmd_n^1 \lmd_n^2 )^{3/2} }  } \\
+ (\lmd_n^1 \lmd_n^2 )^{-\frac{1}{2}} \vm{x_n^1-x_n^2 +\frac{3}{2} [(\lmd_n^1)^3 t^1_n- (\lmd_n^2)^3 t^2_n  ] [ (\xi^1_n)^2 +(\xi^2_n)^2  ]   	}
 \overset{n}{\longrightarrow} \infty. 
\end{multline*}
\item We say that 
$$ (\Gamma_n)_{n \geq 1} =(\lmd_n, \xi_n, x_n, t_n )_{n \geq 1}  \overset{n}{\longrightarrow} \infty, $$ 
if 
$ (\Gamma_n)_{n \geq 1}  $ and $ (1,0,0,0)_{n \geq 1} $ are asymptotically orthogonal , i.e. 
$$
\lmd_n + \frac{1}{ \lmd_n} + \vm{\xi_n} + \vm{t_n} + \vm{x_n}  \overset{n}{\longrightarrow} \infty. 
$$
\end{enumerate}
\end{definition}

Thus one can think of $ \infty $ as an element in the one-point compactification of  $ (0,\infty)\times \mb{R}^3 $. 

\

If $ \Gamma_n^1 =(\lmd_n^1, \xi_n^1, x_n^1, t_n^1 ) $ and  $ \Gamma_n^2 =(\lmd_n^2, \xi_n^2, x_n^2, t_n^2 ) $ are asymptotically orthogonal then 
\be \label{inner.orthog}
\lim_{n \to \infty} \lng g_{x_n^1, \lmd_n^1 } e^{-t^1_n \partial^3_x} [e^{ix \xi_n^1  \lmd_n^1 } \phi ] , \  g_{x_n^2, \lmd_n^2 } e^{-t^2_n \partial^3_x} [e^{ix \xi_n^2  \lmd_n^2 } \varphi ] 
\rng = 0, \qquad \phi, \varphi \in L^2. 
\ee
where either $ \xi_n^j = 0 $ for all $ n \geq 1 $ or $ | \lmd_n^j \xi_n^j | \to \infty $. 
See \cite[Lemma 5.2, 5.1 Cor. 3.7]{shao2009}. 

This implies, in particular, the following statement.  
\begin{lemma} \label{weakgconv}
Let  $ \Gamma_n = (\lmd_n, \xi_n, z_n, s_n ) \to \infty $ and $ \theta_n \in \mathbb{R} $. Then, weakly in $ L^2 $ one has
\be \label{gamma.weak.conv}
e^{i \theta_n}  g_{z_n, \lmd_n} e^{-s_n \partial^3_x} [e^{\pm ix \xi_n  \lmd_n } h ] \rightharpoonup 0,
\ee 
for any $ h \in L^2 $. 
\end{lemma}

%
 
\ 

We are ready to state the profile decomposition for the Airy propagator obtained by Shao in \cite{shao2009}.

\begin{lemma} \label{profile.dec}
(Airy linear profile decomposition \cite{shao2009})
Let $ v_n : \mb{R} \to \mb{R} $ be a sequence of functions bounded in $ L^2_x(\mb{R}) $. Then, after passing to a subsequence, there exist functions $ \phi^j: \mb{R} \to \mb{C} $ in $ L^2_x(\mb{R}) $, group elements $ g_n^j := g_{x_n^j, \lmd_n^j } \in G $, frequency parameters $ \xi_n^j \in [0,\infty) $ and times $ t_n^j \in \mb{R} $ such that for all $ J \geq 1 $ one can write 
\be
v_n = \sum_{1 \leq j \leq J}  g^j_n e^{-t^j_n \partial^3_x} \text{Re}[e^{ix \xi_n^j  \lmd_n^j } \phi^j ]  + w_n^J,
\ee
for some real-valued sequence $ w_n^J $ in $ L^2_x(\mb{R}) $ with 
\be
\lim_{J \to \infty} \limsup_{n \to \infty} \vn{| \partial_x |^{1/6}  e^{-t \partial^3_x} w_n^J  }_{L^6_{t,x}(\mb{R} \times \mb{R} )} = \lim_{J \to \infty} \limsup_{n \to \infty} \vn{e^{-t \partial^3_x} w_n^J }_{L^5_x L^{10}_t (\mb{R} \times \mb{R} )}=0.
\ee
For each $ 1 \leq j \leq J $, the frequency parameters $ \xi_n^j  $ satisfy: either $ \xi_n^j =0 $ for all $ n \geq 1 $ or $ \xi_n^j  \lmd_n^j \to \infty $ as $ n \to \infty $ (If $ \xi_n^j =0 $ we assume $ \phi^j $ is real). For any $ J \geq 1 $ one has
\be \label{L2.orthog}
\vn{v_n}_{L^2}^2- \sum_{1 \leq j \leq J} \vn{\text{Re}[e^{ix \xi_n^j  \lmd_n^j } \phi^j ]  }_{L^2}^2 - \vn{w_n^J}_{L^2}^2  \overset{n}{\longrightarrow} 0. 
\ee
The family of sequences $ \Gamma_n^j =(\lmd_n^j, \xi_n^j, x_n^j, t_n^j ) \in (0,\infty) \times \mb{R}^3 $ are pair-wise asymptotically orthogonal in the sense of Definition \ref{asym.orth} and
for any $ 1 \leq j \leq J $ 
\be \label{orthog2}
\lim_{n \to \infty} \lng g^j_n e^{-t^j_n \partial^3_x} \text{Re}[e^{ix \xi_n^j  \lmd_n^j } \phi^j ] , w_n^J \rng = 0.
\ee
\end{lemma}

For more discussion of the properties stated above we refer to Lemma 2.4, Remark 2.5. in \cite{killip2009mass} and Corollary 3.7, Lemma 5.2 in \cite{shao2009}.

\begin{corollary} \label{cor.profile.dec}
Under the assumptions and notations of Lemma \ref{profile.dec}, if $ v_n \rightharpoonup 0 $ weakly in $ L^2 $, then also $  w_n^J \rightharpoonup 0 $ weakly in $ L^2 $ for all $ J \geq 1 $ after passing to a subsequence. For any $ 1 \leq j \leq J $ one has $ \phi^j=0 $ or $\Gamma_n^j =(\lmd_n^j, \xi_n^j, x_n^j, t_n^j ) \to \infty $ in the sense of Definition \ref{asym.orth} and therefore 
\be
 g^j_n e^{-t^j_n \partial^3_x} \text{Re}[e^{ix \xi_n^j  \lmd_n^j } \phi^j ]  \rightharpoonup 0.
\ee
\end{corollary}

\begin{proof}
After passing to a subsequence, we can arrange so that for each $ j \in \overline{1,J} $, either $ \Gamma_n^j  $ converges to a finite $ \Gamma_0^j $ in $ (0,\infty) \times \mb{R}^3 $ or $ \Gamma_n^j  \to \infty $. By pair-wise asymptotic orthogonality we have \eqref{inner.orthog} and therefore at most one of the sequences $ \{ ( \Gamma_n^j)_{n \geq 1} \ | \ 1 \leq j \leq J, \ \phi^j \neq 0 \}  $ can converge to a finite value. Assume this happens for $ j=1 $ and then $ \xi^1_n =0 $ for all $ n \geq 1 $ and $ \phi^1 $ is assumed real.  Since $ v_n \rightharpoonup 0 $ we obtain 
$$  g^1_0 e^{-t^1_0 \partial^3_x} \phi^1  +   w_n^J \rightharpoonup 0.  
$$ 
Taking inner product with $ g^1_n e^{-t^1_n \partial^3_x} \phi^1 $ and using \eqref{orthog2} we obtain $ \vn{g^1_0 e^{-t^1_0 \partial^3_x} \phi^1}_{L^2}^2=0 $ and then $ \phi^1=0 $, which is a contradiction. \end{proof}

\

Finally, we recall the decoupling property of nonlinear profiles proved in \cite[Lemma 2.6]{killip2009mass}. When $ \xi_n \lmd_n \to \infty $ the decoupling will follow from this lemma together with the approximation \eqref{smoothapprox} from Theorem \ref{oscilatory.profiles}

\begin{lemma}[\cite{killip2009mass}] \label{decoupling}
Let $ \psi^1, \psi^2 \in C_c^{\infty} (\mb{R} \times \mb{R}) $ and sequences 
$$ (\Gamma_n^1)_{n \geq 1} =(\lmd_n^1, \xi_n^1, x_n^1, t_n^1 )_{n \geq 1}, \quad  (\Gamma_n^2)_{n \geq 1} =(\lmd_n^2, \xi_n^2, x_n^2, t_n^2 )_{n \geq 1}, $$ in $ (0,\infty) \times \mb{R}^3 $ assumed 
asymptotically orthogonal in the sense of Definition \ref{asym.orth}. Then one has:
$$
\lim_{n \to \infty}  \vn{ T_{g_{x^1_n, \lmd^2_n}} \psi^1(t+t_n^1) \  T_{g_{x^2_n, \lmd^2_n}} \psi^2(t+t_n^2)}_{L^{\frac{5}{2}}_x L^5_t }  =0, 
$$
in the case $ \xi^1_n \equiv \xi^2_n \equiv 0 $, and 
$$
\lim_{n \to \infty}  \vn{ T_{g_{x^1_n, \lmd^1_n}} \lpp \psi^1( 3 \lmd^1_n \xi^1_n (t+t_n^1), x+ 3 (\lmd^1_n \xi^1_n)^2 (t+t_n^1)  ) \rpp   T_{g_{x^2_n, \lmd^2_n}} \psi^2(t+t_n^2)}_{L^{\frac{5}{2}}_x L^5_t }  =0, 
$$ 
when $ \xi^1_n \lmd^1_n \to \infty $ and $ \xi^2_n \equiv 0 $, while
\begin{multline*}
\lim_{n \to \infty}  \vn{ 
T_{g_{x^1_n, \lmd^2_n}} \lpp \psi^1( 3 \lmd^1_n \xi^1_n (t+t_n^1), x+ 3 (\lmd^1_n \xi^1_n)^2 (t+t_n^1)  ) \rpp   \\
T_{g_{x^2_n, \lmd^2_n}} \lpp \psi^2( 3 \lmd^2_n \xi^2_n (t+t_n^2), x+ 3 (\lmd^2_n \xi^2_n)^2 (t+t_n^2)  ) \rpp   
}_{L^{\frac{5}{2}}_x L^5_t }  =0, 
\end{multline*}
when $ \xi^1_n \lmd^1_n \to \infty $ and $ \xi^2_n \lmd^2_n \to \infty $.
\end{lemma}

\section{Reduction to an almost periodic solution - Proof of Theorem \ref{almost_periodic_prop}} \label{Sec.reduction.to.almost.periodic}

This section is devoted to the proof of Theorem \ref{almost_periodic_prop}. Therefore we will assume at least one $ \delta $-close solution exists. Then we define the set 
$$
S(\delta) \defeq \{ u \ | \ u = \text{solution} \ \delta-close \ \text{to} \  Q \ \text{with} \  M(u) < M_Q   \}
$$
and the minimal mass:
$$
m_0(\delta) \defeq \inf \{  M(u) \ | u \in  S(\delta) \}.
$$
By the triangle inequality, if $ u \in S(\delta) \neq \emptyset $ and $ t_0 \in I $ we have the basic bounds
\be \label{mass.closeness}
M_Q^{\frac{1}{2}} -\delta \leq \vn{ u(t_0)}_{L^2_x} < M_Q^{\frac{1}{2}}, \quad \text{and} \quad   M_Q^{\frac{1}{2}} -\delta  \leq m_0^{\frac{1}{2}} \leq M_Q^{\frac{1}{2}}.
\ee

\

The crux of the proof is the following Palais-Smale -type proposition which is used
to extract subsequences convergent in $ L^2 $.

\begin{proposition} \label{L2convergence_prop}
There exists an $ \delta>0 $ small enough such that the following holds. 
Let $ u_n:I_n \times \mb{R} \to \mb{R} $ be maximal-lifespan (strong) solutions to the mass-critical focusing gKdV equation \eqref{1.1} which are $ \delta $-close to $ Q $, i.e. for some continuous $ g_n : I_n \to G $ one has
\be \label{epcloseseq}
 \vn{ g_n(t) u_n(t) - Q}_{L^2} \leq \delta \qquad \forall \ t \in I_n, \ n \geq 1. 
\ee
Suppose $ M(u_n) \searrow m_0=m_0(\delta) $ and let $ t_n \in I_n $ be a sequence of times. Then the sequence  $ g_n(t_n) u_n(t_n) $ has a subsequence which converges in $ L^2 $ to a function $ \phi $ with $ M(\phi) =m_0 $. 
\end{proposition}

Assuming Proposition \ref{L2convergence_prop} we can now construct almost periodic solutions.

\begin{proof}[{\bf Proof of Theorem \ref{almost_periodic_prop}}]
We first show that if $ M(u)>m_0 $ then there exists a maximal-lifespan solution $ v : J \times \mb{R} \to \mb{R} $ with minimal mass $ M(v)=m_0 $ which is $ \delta $-close to $ Q $. In that case there exists a sequence of maximal-lifespan solutions $ u_n:I_n \times \mb{R} \to \mb{R} $ with $ M(u_n) \searrow m_0 $ such that \eqref{epcloseseq} holds for some continuous $ g_n : I_n \to G $. Then we apply Prop. \ref{L2convergence_prop} with some $ t_n \in I_n $ and obtain a $ \phi \in L^2 $ with $ \vn{\phi}_{L^2}=m_0^{1/2} $. By translating time we may assume all $ t_n=0 $ and by applying transformations $ T_{g_n(0)^{-1}} $ we may assume without loss of generality that all $ g_n(0) $ are the identity. 
Let $ v $ be the strong solution to  \eqref{1.1} with initial data $ v(0) =\phi $, defined on a maximal interval $ J $, which then satisfies 
$$ \vn{ u_n(0) - v(0)}_{L^2_x} \to 0. $$

Then for any $ t \in J $, by continuous dependence on initial data, see Corollary  \ref{sol.conv} applied on $ [0,t] $, one has $ t \in I_n $ for $ n $ large enough and 
\be \label{un.v.conv} 
\vn{ u_n(t) - v(t)}_{L^2_x} + \vn{u_n-v}_{ L^5_x L^{10}_t([0,t] \times \mb{R}) }  \to 0. 
\ee
By Lemma \ref{group.compactness} we have $ g_n(t) \in K_t $ for a compact set $ K_t $. Then we can extract a subsequence such that $ g_n(t) $ converges to some $ g(t) \in G $ in the strong operator topology. Therefore \eqref{un.v.conv}  and \eqref{epcloseseq} imply
$$  \vn{ g(t) v(t) - Q}_{L^2} \leq \delta \qquad \forall \ t \in J, $$
which gives the desired $ \delta $-closeness to $ Q $. Note that $ g(t) $ is continuous. 
\

We now show that $ v $ is almost periodic modulo symmetries. This follows by considering a new arbitrary sequence of times $ t_n \in J $ and applying Prop. \ref{L2convergence_prop}  with $ g_n =g $, $ u_n =v $ and $ t_n \in I_n=J $ to conclude that $ g(t_n) v(t_n) $ has a limit point in $ L^2 $. 
\end{proof}

\

It remains to prove the key convergence result.

\begin{proof}[{\bf Proof of Proposition \ref{L2convergence_prop}}] By translating time we may assume all $ t_n=0 $ and by applying transformations $ T_{g_n(0)^{-1}} $ we may assume without loss of generality that all $ g_n(0) $ are the identity. 

We divide the proof into several steps and for the first steps we largely follow the outline of \cite[Prop. 5.1 -Case II]{killip2009mass}, with the mention that here one needs to insure that the bulk of $ m_0 $, except for $ O(\delta) $ mass, has to fall onto the first profile. 

{\bf Step 1.} (Decomposing the sequence) 

By passing to a subsequence, using the Banach-Alaoglu theorem, we obtain a function $ \phi^1 \in L^2 $ such that $ u_n(0) \rightharpoonup \phi^1 $ weakly in $ L^2 $. Note that $ \vn{\phi^1}_{L^2}^2 \leq m_0 $ and since $ u_n(0) - Q \rightharpoonup \phi^1 -Q $ we obtain 
\be
\vn{\phi^1 -Q}_{L^2} \leq \delta. 
\ee 
Moreover, 
\be \label{unphi.conv}
\vn{ u_n(0) - \phi^1}_{L^2}^2=\vn{u_n(0) }_{L^2}^2+ \vn{\phi^1}_{L^2}^2- 2 \lng u_n(0), \phi^1 \rng  \overset{n}{\longrightarrow}  m_0-\vn{\phi^1}_{L^2}^2
\ee 
If $ \vn{\phi^1}_{L^2}^2 = m_0 $ this implies the desired convergence. Now assume $ \vn{\phi^1}_{L^2}^2 < m_0 $ and we will obtain a contradiction. We use the profile decomposition in Lemma \ref{profile.dec} and its Corollary \ref{cor.profile.dec} applied to $ v_n=u_n(0) - \phi^1  \rightharpoonup 0 $ to write for any $ J \geq 2 $
$$
u_n(0) - \phi^1 =  \sum_{2 \leq j \leq J}  g^j_n e^{-t^j_n \partial^3_x} \text{Re}[e^{ix \xi_n^j  \lmd_n^j } \phi^j ]  + w_n^J. 
$$
By \eqref{unphi.conv}, the limit \eqref{L2.orthog} becomes
\be \label{L2.orthog2}
m_0-\vn{\phi^1}_{L^2}^2- \sum_{2 \leq j \leq J} \vn{\text{Re}[e^{ix \xi_n^j  \lmd_n^j } \phi^j ]  }_{L^2}^2 - \vn{w_n^J}_{L^2}^2   \overset{n}{\longrightarrow}  0. 
\ee
By re-denoting some indices, we may assume that all the $ \phi^j $'s are nonzero. Defining  $\Gamma_n^1 =(1, 0, 0, 0 ) $ corresponding to $ \phi^1 $, from Corollary  \ref{cor.profile.dec} we obtain that $ \Gamma_n^j =(\lmd_n^j, \xi_n^j, x_n^j, t_n^j ) \to \infty $ for $ j \geq 2 $,  and thus all $ (\Gamma_n^j)_{j \geq 1} $ are pair-wise asymptotically orthogonal and
\be \label{wnJ.weaklimit}
 w_n^J \rightharpoonup 0.
\ee
From \eqref{mass.closeness} and $ \vn{\phi^1}_{L^2}^2 \leq m_0 $ we obtain the smallness condition 
\be \label{smallness}
\sum_{2 \leq j \leq J} \vn{\text{Re}[e^{ix \xi_n^j  \lmd_n^j } \phi^j ]  }_{L^2}^2 + \vn{w_n^J}_{L^2}^2 < 2 \delta M_Q^{\frac{1}{2}}, \qquad \forall \ n \gg_J 1.
\ee

\

{\bf Step 2.} (Construct nonlinear profiles)
\

Let $ v^1 :I \times \mb{R} \to \mb{R} $ be the maximal-lifespan solution to \eqref{1.1} with initial data $ v^1(0)=\phi $. 
We continue with defining solutions associated to the profiles for $ j \geq 2 $. For each $ J \geq 2 $ we reorder the indices such that:

A) For  $ j \in \overline{2,J_0} $ one has $ \xi_n^j \equiv 0 $. Then one can refine the sequence for each $ j $ and by a diagonal argument one can assume that each sequence $ (t^j_n)_{n \geq 1} $ has a limit $ T^j $, possibly $ \pm \infty $. If $ T^j $ is finite one may assume that $ t^j_n \equiv T^j =0 $ by replacing $ \phi^j $ by $ e^{T^j \partial_x^3} \phi^j $ and by absorbing $ e^{- (t^j_n-T^j) \partial_x^3} \text{Re}\phi^j -  \text{Re}\phi^j $ into the remainder term $ w^J_n $. One defines:
\begin{itemize}
\item When  $ t^j_n \equiv 0 $, let $ v^j $ be the the maximal-lifespan solution to \eqref{1.1} with $ v^j(0)=\text{Re}\phi^j $.
\item If $  t^j_n \to \pm \infty $, let $ v^j $ be the the maximal-lifespan solution to \eqref{1.1} which scatters forward/backward in time to  $ e^{- t \partial_x^3} \text{Re}\phi^j  $.
\end{itemize}  
Due to the smallness property \eqref{smallness}, each $ v^j $ is global and $ S_{\mb{R}} (v^j) \ls M[	\text{Re}\phi^j] $.

The nonlinear profiles are defined by
$$
v^j_n(t) \defeq T_{g^j_n}[ v^j(\cdot + t^j_n)](t), \qquad  j \in \overline{2,J_0}, \ n \geq 1,
$$ 
so that $ v^j_n : \mb{R} \times \mb{R} \to \mb{R} $ with $ v^j_n(0)=g^j_n v^j(t_n) $. 

B) For $ j \in \overline{J_0+1,J} $ the reordering satisfies $  \xi_n^j \lmd_n^j \to \infty $. 
For $ n $ sufficiently large, the solution to \eqref{1.1} with data
$$ 
\tilde{v}^j_n (t^j_n) = e^{-t^j_n \partial^3_x} \text{Re}[e^{ix \xi_n^j  \lmd_n^j } \phi^j ]
$$
is global and small. Moreover, by applying the Riemann-Lebesgue lemma to 
$$ 
2 \vn{\text{Re}[e^{ix \xi_n^j  \lmd_n^j } \phi^j ]  }_{L^2}^2 = M(\phi^j) +\int_{\mb{R}}
 \text{Re}[e^{i 2 x \xi_n^j  \lmd_n^j } \phi^j(x)^2 ]  \dd x
$$
to obtain a bound on $ M(\phi^j) $, 
one has the approximation given by Theorem \ref{oscilatory.profiles} 
(since one can insure, using a diagonal argument, that $ (t^j_n \xi^j_n \lmd^j_n)_{n \geq 1} $ has a limit).  

Again, one transforms these solutions to obtain $ v^j_n : \mb{R} \times \mb{R} \to \mb{R} $ by
$$
v^j_n(t) \defeq T_{g^j_n}[ \tilde{v}^j_n(\cdot + t^j_n)](t), \qquad  j \in \overline{J_0+1, J}, \ n \gg 1.
$$

\

For both cases A) and B) Lemma \ref{decoupling} and Theorem \ref{oscilatory.profiles} give the decoupling property 
\be \label{decouplingg}
\lim_{n \to \infty}  \vn{v^j_n v^k_n}_{L^{\frac{5}{2}}_x L^5_t (I \times \mb{R})}  =0 \qquad \qquad \forall \ 1 \leq j <k 
\ee
where for $ j=1 $ we denote $ v^1_n=v^1 $. 

Moreover, due to the smallness and the invariance of the scattering norm one has
\be \label{vjn.bound}
S_{\mb{R}} (v^j_n) \ls \vn{\text{Re}[e^{ix \xi_n^j  \lmd_n^j } \phi^j ]}_{L^2_x}^2, \qquad  j \geq 2, \ n \gg_j 1.
\ee

\

{\bf Step 3.} (Construct approximate solutions and bound the difference)

For any $ J \geq 2 $ construct the approximate solution, defined on $ I $ for $ n \gg_J 1 $ by
$$ \tilde{u}_n^J(t) \defeq v^1(t)+\sum_{j=2}^J v^j_n(t) + e^{-t \partial^3_x} w^J_n. $$
and define the remainders $ r_n^J $ on $ I \cap I_n  $ by
$$ u_n(t)= \tilde{u}_n^J(t) + r_n^J(t). $$

From the way the $ v^j_n $ were constructed we obtain
\be \label{rn0}
\vn{r_n^J(0)} _{L^2}= \vn{u_n(0)- \tilde{u}_n^J(0)}_{L^2} \overset{n}{\longrightarrow} 0,  \qquad \forall J \geq 2.
\ee
Next we bound the scattering size on any interval $ \tilde{I} $, using \eqref{smallness},\eqref{vjn.bound} and using the decoupling \eqref{decouplingg} after having raised the sum to the power 5:
\begin{align} \nonumber
\limsup_{n \to \infty} S_{\tilde{I}}( \tilde{u}_n^J ) & \ls \limsup_{n \to \infty} S_{\tilde{I}}( \sum_{j=1}^J v^j_n ) + \limsup_{n \to \infty} S_{\mb{R} }(e^{-t \partial^3_x} w^J_n) \\
\nonumber
& \ls S_{\tilde{I}}(v^1) +  \limsup_{n \to \infty}  \sum_{j=2}^J S_{\mb{R} }(v^j_n ) + \delta \\
\label{Scatteringlimit}
& \ls S_{\tilde{I}}(v^1) +  \limsup_{n \to \infty}  \sum_{j=2}^J \vn{\text{Re}[e^{ix \xi_n^j  \lmd_n^j } \phi^j ]}_{L^2_x}^2+ \delta \ls S_{\tilde{I}}(v^1) + \delta.
\end{align}

\

In the remainder of this step we prove 
\be \label{limit.intervals} 
I \subseteq  \bigcup_{N \geq 1} \bigcap_{n \geq N} I_n  
\ee
and that for any $ t \in I $ one has
\be \label{rnt}
\lim_{J \to \infty} \limsup_{n \to \infty}   \vn{r_n^J(t)} _{L^2}=0. 
\ee 
Suppose $ t > 0 $. Divide $ [0,t] $ into intervals $[t_k, t_{k+1}] $, $ k \in \overline{1,N} $, $ t_1 = 0 $ such that 
\be \label{v1.ep0}
\vn{v^1}_{L^5_x L^{10}_t (  [t_k, t_{k+1}]  \times \mb{R} ) } \simeq \ep_0 \qquad \forall \ k \in \overline{1,N-1} 
\ee
where $ \ep_0=\ep_0(M_Q,1)>0 $ is the universal constant given by Lemma \ref{short.time.stability}. Then Lemma \ref{num.intervals} gives a bound on the number of intervals $ N $. 

We begin with \eqref{rn0} and do an inductive argument to show that if $ t_k \in I_n $ for $ n \gg_k 1 $  and 
\eqref{rnt} holds at $ t=t_k $, then $ t_{k+1} \in I_n $ holds for $ n \gg_{k+1} 1 $ and 
$$ 
\lim_{J \to \infty} \limsup_{n \to \infty}   \vn{r_n^J} _{L^{\infty}_t L^2_x ([t_k,t_{k+1}] \times \mb{R})}=0. 
$$
These facts follow from the short-time stability Lemma \ref{short.time.stability} applied with $ u_n $ and $ \tilde{u}_n^J $, provided we check:
\begin{align}
\label{sclim}
 \limsup_{n \to \infty}  \vn{ \tilde{u}_n^J}_{L^5_x L^{10}_t(  [t_k, t_{k+1}]  \times \mb{R} ) } \leq \frac{ \ep_0}{2} \qquad \forall \ J \geq 2, \ k \in \overline{1,N} \\
\label{l1l2lim}
\lim_{J \to \infty} \limsup_{n \to \infty} \vn{ \vm{\partial_x}^{-1}  [ (\pt_t + \partial_x^3)\tilde{u}_n^J- \partial_x ( \tilde{u}_n^J)^5 ] }_{L^1_x L^2_t ([t_k, t_{k+1}]  \times \mb{R})} = 0.
\end{align} 
The first bound \eqref{sclim} follows from \eqref{Scatteringlimit} by appropriately choosing the implicit constant in \eqref{v1.ep0}  and choosing $ \delta $ small enough. 
The asymptotic solution bound \eqref{l1l2lim} is proved in Lemma \ref{approx.sol} below.
This completes the proof of \eqref{limit.intervals} and \eqref{rnt}. Moreover, by summing over intervals and recalling that $ \ep_0 $ is fixed, this argument and Lemma \ref{short.time.stability} give the uniform bound  
\be \label{Scat.norm.unif}
\vn{u_n}_{L^5_x L^{10}_t([0,t]) } \ls N \ep_0 \leq C \big(  \vn{v^1}_{L^5_x L^{10}_t([0,t]) } \big), \qquad n \gg_t 1. 
\ee

\

{\bf Step 4.}  (Show that $  v^j_n(t) $ converges weakly to $ 0 $) 

Fix $ t \in \mb{R} $ and $ j \geq 2 $. Recall that $\Gamma_n^j =(\lmd_n^j, \xi_n^j, x_n^j, t_n^j ) \to \infty $ in the sense of Definition \ref{asym.orth}. 

A) We first assume $ \xi^j_n \equiv 0 $. Then
$$ v^j_n(t) = g^j_n v^j \Big( t^j_n + \frac{t}{(\lmd^j_n)^3} \Big). $$
 By passing to a subsequence, we may assume 
$$ t^j_n + \frac{t}{(\lmd^j_n)^3} \to T_j \in [-\infty,\infty].  $$ 
If $ T_j $ is finite, in either case $ t^j_n \equiv 0 $ or $ t^j_n \to \infty $ we have $ g^j_n \to \infty $ and the claim reduces to $ g^j_n v^j(T_j) \rightharpoonup 0 $, which follows from \eqref{gamma.weak.conv}. 

If $ T_j \to \pm \infty $ we use scattering to replace $ v^j_n(t) $ by $  g^j_n e^{-\big( t^j_n + \frac{t}{(\lmd^j_n)^3} \big) \partial^3_x} v_{\pm} $. Then we can approximate by bump functions and apply the dispersive estimate. 

B) It remains to consider the case $ \xi^j_n \lmd^j_n \to \infty $. This implies in particular that $ \xi^j_n + \lmd^j_n \to \infty    $. Fix $ t \in \mb{R} $, $ \epsilon >0 $, $ j \geq J_0 +1 $ and $ \varphi \in C_c^{\infty}(\mb{R}) $. We will use the approximation involving NLS solutions from Theorem \ref{oscilatory.profiles} to show
$$
\vm{ \lng v_n^j (t), \varphi \rng } < \epsilon
$$ 
for $ n $ large enough. Since
$$ 
v_n^j (t) = g_n^j \tilde{v}_n^j \Big(t_n^j + \frac{t}{(\lmd_n^j)^3} \Big) 
$$
we can use the approximation \eqref{unifapproxu} to reduce to 
$$
\vm{ \lng g_n^j  \tilde{u}_n^T \Big( t_n^j + \frac{t}{(\lmd_n^j)^3} \Big) , \varphi \rng } < \frac{\epsilon}{2}
$$
for a fixed large $ T $, where the $ \tilde{u}_n^T $ are defined by \eqref{unt} in terms of NLS solutions $ V_n $. 

By passing to a subsequence we may assume that all the $ t_n^j + \frac{t}{(\lmd_n^j)^3} $ are in $ [ -\frac{T}{3 \xi_{n} \lmd_n}, \frac{T}{3 \xi_{n} \lmd_n} ] $ or in $ [ \frac{T}{3 \xi_{n} \lmd_n} , \infty) $ or in $ (-\infty, -\frac{T}{3 \xi_{n} \lmd_n}] $ and that in the first case we have a limit 
$$ T_1 \defeq  \lim_{n \to \infty} 3 \xi_{n} \lmd_n \Big( t_n^j + \frac{t}{(\lmd_n^j)^3} \Big) \in [-T, T]. 
$$
In the other two cases we define $ T_1 \defeq \pm T $. Using \eqref{VnV}, \eqref{unt} and $ V \in C_t L_x^2 $ we approximate
$$
\vn{  \tilde{u}_n^T \Big( t_n^j + \frac{t}{(\lmd_n^j)^3} \Big) - f_n(T_1) }_{L^2} < \epsilon^2,  \qquad n \gg 1
$$
where we denote
$$
 f_n (T_1) \defeq e^{-s_n \partial^3_x} \text{Re} [e^{ ix \xi_n  \lmd_n } e^{i c_n} V(T_1, x-y_n) ] 
$$
for some values $ s_n, c_n, y_n $. 
Therefore, denoting $ W $ to be either $ V $ or $ \bar{V} $, we reduce to showing
$$
e^{\pm i \theta_n} g_{x_n^j, \lmd_n^j } g_{y_n, 1}  e^{-s_n \partial^3_x}   [e^{ \pm ix \xi_n  \lmd_n } W(T_1) ]  \rightharpoonup 0,
$$
for some $ \theta_n $'s. This follows from Lemma \ref{weakgconv} because $  g_{x_n^j, \lmd_n^j } g_{y_n, 1} = g_{z_n, \lmd_n^j} $ for some $ z_n $ and we have $\xi^j_n +  \lmd^j_n \to \infty $.

\

From A) and B) we conclude
\be \label{vjn.weak.conv}
v^j_n(t) \rightharpoonup 0, \qquad \forall \  t \in \mb{R}, \ j \geq 2. 
\ee

\

{\bf Step 5.} (Prove that $ v^1 $ is $ \delta $-close to $ Q $) 

Fix an arbitrary $ t \in I $, where we recall that $ I $ is the maximal lifespan of $ v^1 $. Then, by \eqref{limit.intervals} we have  $ t \in I_n $ for $ n $ large enough. We expand
\be \label{expand}
\delta^2 \geq \vn{u_n(t) - g_n(t)^{-1} Q }_{L^2}^2 = \vn{v^1(t)-g_n(t)^{-1} Q }_{L^2}^2 + A_n^J(t) + B_n^J(t)
\ee
with the terms 
\begin{align*}
A_n^J(t) & \defeq \vn{ \sum_{j=2}^J v^j_n(t) + e^{-t \partial^3_x} w^J_n+ r_n^J(t)}_{L^2}^2 \\
B_n^J (t) & \defeq 2 \lng v^1(t)-g_n(t)^{-1} Q \ ,  \sum_{j=2}^J v^j_n(t) + e^{-t \partial^3_x} w^J_n+ r_n^J(t)		\rng.
\end{align*}
Due to the uniform bound \eqref{Scat.norm.unif}, Lemma \ref{group.compactness} provides the existence of a compact set $ K_t $ such that $ g_n(t) \in K $ for $ n $ large enough. We extract a subsequence such that $ g_n(t) $ converges to some $ g(t) \in G $ in the strong operator topology. Then also $ g_n(t)^{-1} \to g(t)^{-1} $, so we may replace $ g_n(t)^{-1} Q  $ by $ g(t)^{-1} Q $ when we use  \eqref{vjn.weak.conv},  \eqref{wnJ.weaklimit} and \eqref{rnt} to obtain 
$$
\lim_{J \to \infty} \limsup_{n \to \infty} B_n^J (t)=0.
$$
We use this together with $ A_n^J(t) \geq 0 $ to pass to the limit in \eqref{expand} and conclude
$$ \vn{g(t)  v^1(t)-Q }_{L^2} \leq \delta \qquad \forall t \in I.
$$
This means $ v^1 \in S(\delta) $ with $ M(v^1) < m_0(\delta) $, a contradiction. 
\end{proof}

\

It remains to verify the asymptotic solution bound \eqref{l1l2lim}. 

\begin{lemma} \label{approx.sol}
Suppose $ w_n^J \in L^2_x(\mb{R}) $, $ J \geq 1,\ n \geq 1 $ and that 
$ v^j_n \in L^5_x L^{10}_t (\tilde{I} \times \mb{R} ) $ are solutions to \eqref{1.1}  such that for any $ 1 \leq j <k $ 
$$
\lim_{n \to \infty}  \vn{v^j_n v^k_n}_{L^{\frac{5}{2}}_x L^5_t (\tilde{I} \times \mb{R})}  =0,  \qquad     \lim_{J \to \infty} \limsup_{n \to \infty} \vn{e^{-t \partial^3_x} w_n^J }_{L^5_x L^{10}_t (\tilde{I} \times \mb{R} )}=0.
$$
Then, assuming the $ \tilde{u}_n^J $ are uniformly bounded in  $ L^5_x L^{10}_t (\tilde{I} \times \mb{R} ) $, defined by  
$$ \tilde{u}_n^J(t) \defeq \sum_{j=1}^J v^j_n(t) + e^{-t \partial^3_x} w^J_n, $$
one has
$$
\lim_{J \to \infty} \limsup_{n \to \infty} \vn{ \vm{\partial_x}^{-1}  [ (\pt_t + \partial_x^3)\tilde{u}_n^J- \partial_x ( \tilde{u}_n^J)^5 ] }_{L^1_x L^2_t ( \tilde{I} \times \mb{R})} = 0.
$$
\end{lemma}

\begin{proof}
This is proved in \cite[Lemma 5.3]{killip2009mass}. We review the argument for the sake of completeness. One writes 
 $$
\left(\partial_{t}+\partial_{x}^3\right) \tilde{u}_{n}^{J}=\sum_{1 \leq j \leq J} \partial_{x}\left(v_{n}^{j}\right)^{5}.
$$
Thus it suffices to estimate $ (\tilde{u}_{n}^{J})^{5}-\sum_{1 \leq j \leq J} (v_{n}^{j})^{5} $ as follows:
$$
\vn{ \left(\tilde{u}_{n}^{J}-e^{-t \partial_{x}^{3}} w_{n}^{J}\right)^{5}-\left(\tilde{u}_{n}^{J}\right)^{5}}_{L^1_x L^2_t ( \tilde{I} \times \mb{R})} \ls 
\vn{(e^{-t \partial_{x}^{3}} w_{n}^{J})^{5} }_{L^1_x L^2_t ( \tilde{I} \times \mb{R})} +
\vn{ (e^{-t \partial_{x}^{3}} w_{n}^{J}) \left|\tilde{u}_{n}^{J}\right|^{4}}_{L^1_x L^2_t ( \tilde{I} \times \mb{R})},
$$
then one uses Holder's inequality and pass to the limit. Secondly,
$$
\vn{ \big( \sum_{1 \leq j \leq J} v_{n}^{j}\big)^{5}-\sum_{1 \leq j \leq J}\left(v_{n}^{j}\right)^{5} }_{L^1_x L^2_t ( \tilde{I} \times \mb{R})} \ls 
\sum_{i_{1}, i_{2}, i_{3}=1}^{J} \sum_{1 \leq j \neq k \leq J} \vn{ v_{n}^{i_{1}} v_{n}^{i_{2}} v_{n}^{i_{3}} ( v_{n}^{j} v_{n}^{k} )}_{L^1_x L^2_t ( \tilde{I} \times \mb{R})},
$$
and one uses Holder's inequality again to pass to the limit.
This completes the proof.
\end{proof}

\section{Reductions of an almost periodic solution}

Having proved Theorem \ref{almost_periodic_prop}, we have reduced the main result, Theorem \ref{main.thm}, to the case of almost periodic solutions.  
The remainder of the paper is devoted to this case, i.e. proving Theorem \ref{t1.1}.  
We begin with studying $ N(t) $ from Definition \ref{d1.1}. In this section we prove

\begin{theorem}\label{t3.3}
If there exists an almost periodic solution to $(\ref{1.1})$ with $\| u_0 \|_{L^{2}} < \| Q \|_{L^{2}}$, then there exists an almost periodic solution to \eqref{1.1} satisfying \eqref{1.4} on a maximal interval $I$ with $N(t) \geq 1$ on $I$, and
\begin{equation}\label{3.33}
\int_{I} N(t)^{2} dt = \infty.
\end{equation}
Moreover, if the initial solution is $ \delta $-close to $ Q $, then the solution we obtain is also $ \delta $-close to $ Q $. 
\end{theorem}

\begin{proof}[{\bf Proof of Theorem \ref{t3.3}}]

Using elementary reductions (see \cite{killip2009mass}) it suffices to consider an almost periodic solution to $(\ref{1.1})$ that satisfies $N(t) \leq 1$ for $t \in [0, \infty)$. Such a solution will satisfy one of two properties:
\begin{equation}\label{3.1}
\lim_{T \rightarrow \infty} \inf_{t \in [0, T]} N(t) > 0,
\end{equation}
or
\begin{equation}\label{3.2}
\lim_{T \rightarrow \infty} \inf_{t \in [0, T]} N(t) = 0.
\end{equation}

{\bf 1)} Begin with scenario $(\ref{3.1})$, $N(t) \sim 1$ for any $t \in [0, \infty)$. Thus, there exists a function $x(t) : [0, \infty) \rightarrow \mathbb{R}$ such that
\begin{equation}\label{3.3}
\{ u(t, x - x(t)) : t \in [0, \infty) \}
\end{equation}
lies in a precompact subset of $L^{2}(\mathbb{R})$. Therefore, taking $t_{n} \rightarrow +\infty$ and possibly after passing to a subsequence,
\begin{equation}\label{3.4}
u(t_{n}, x - x(t_{n})) \rightarrow u_{0} \qquad \text{in} \qquad L^{2}(\mathbb{R}),
\end{equation}
and moreover, $u_{0}$ is the initial data for a solution to $(\ref{1.1})$ satisfying
\begin{equation}\label{3.5}
\{ u(t, x - x(t)) : t \in \mathbb{R} \}
\end{equation}
lies in a precompact subset of $L^{2}(\mathbb{R})$.\medskip

{\bf 2)} Now consider scenario $(\ref{3.2})$. Split this scenario into two separate cases:
\begin{equation}\label{3.6}
\limsup_{T \rightarrow \sup(I)} \frac{\sup_{t \in [t_{0}(T), T]} N(t)}{N(t_{0}(T))} < \infty,
\end{equation}
or
\begin{equation}\label{3.7}
\limsup_{T \rightarrow \sup(I)} \frac{\sup_{t \in [t_{0}(T), T]} N(t)}{N(t_{0}(T))} = \infty.
\end{equation}
where
$$
t_{0}(T)=\inf \big\{t \in[0, T]: N(t)=\inf _{t \in[0, T]} N(t)\big\}
$$

Following \cite{dodson2017global}, for any $k \in \mathbb{Z}$, let
\begin{equation}\label{3.8}
t_{k} = \inf \{ t \in [0, T] : N(t) = 2^{-k} \}.
\end{equation}
Since $N(t)$ is a continuous function of time and $(\ref{3.2})$ holds, $t_{k}$ is well-defined. 

{\bf 2A)} When $(\ref{3.6})$ holds, there exists $C < \infty$ such that $N(t) \leq C 2^{-k}$ for any $t \geq t_{k}$.
\begin{lemma}\label{t3.1}
Suppose $(\ref{3.2})$ and $(\ref{3.6})$ hold.
Then the sequence $(t_{k + 1} - t_{k}) \cdot 2^{-3k}$ is unbounded as $k \rightarrow +\infty$.
\end{lemma}
\noindent \emph{Proof:} Suppose that there exists a constant $C_{0}$ such that
\begin{equation}\label{3.9}
(t_{k + 1} - t_{k}) \cdot 2^{-3k} \leq C_{0}.
\end{equation}
Then for any $k \in \mathbb{Z}$,
\begin{equation}\label{3.10}
t_{k} \lesssim C_{0} 2^{3k}.
\end{equation}
Meanwhile, as in  the scaling symmetry implies
\begin{equation}\label{3.11}
t_{k} \gtrsim 2^{3k}.
\end{equation}
Therefore, for any $k$,
\begin{equation}\label{3.12}
N(t_{k}) \sim t_{k}^{-1/3}.
\end{equation}
As in \cite{dodson2017global}, $(\ref{3.12})$ implies that after passing to another subsequence, we have a solution $u$ to $(\ref{1.1})$ satisfying $N(t) \sim t^{-1/3}$ for any $t \geq 0$. Moreover, following the exact arguments in Section five of \cite{dodson2017global} shows that the self similar solution $u(t,x)$ satisfies the estimate
\begin{equation}\label{3.13}
E(u) \lesssim 1.
\end{equation}
However, by the Gagliardo-Nirenberg inequality, this contradicts $N(t) \nearrow +\infty$ as $t \searrow 0$. $\Box$\medskip

Now take a sequence $t_{k} \rightarrow \infty$ such that
\begin{equation}\label{3.14}
(t_{k + 1} - t_{k}) \cdot 2^{-3k} \rightarrow +\infty.
\end{equation}
In this case, $(\ref{3.6})$ guarantees that $N(t) \sim 2^{-k}$ for any $t_{k} < t < t_{k + 1}$. Choose the sequence of times $t_{k}' = \frac{t_{k} + t_{k + 1}}{2}$. After passing to a subsequence,
\begin{equation}\label{3.15}
2^{k/2} u(t_{k}', 2^{k}(x - x(t_{k}'))) \rightarrow u_{0}, \qquad \text{in} \qquad L^{2}(\mathbb{R}),
\end{equation}
and furthermore, $u_{0}$ is the initial data of a solution to $(\ref{1.1})$ satisfying
\begin{equation}\label{3.16}
\{ u(t, x - x(t)) : t \in \mathbb{R} \}
\end{equation}
lies in a precompact subset of $L^{2}(\mathbb{R})$.\medskip

{\bf 2B)} Finally, consider the case when $(\ref{3.2})$ and $(\ref{3.7})$ hold. In this case, possibly after passing to a subsequence,
\begin{equation}\label{3.17}
2^{k/2} u(t_{k}, 2^{k}(x - x(t_{k}))) \rightarrow u_{0}, \qquad \text{in} \qquad L^{2}(\mathbb{R}),
\end{equation}
where $u_{0}$ is the initial data of a solution to $(\ref{1.1})$ on an interval $I$ such that
\begin{equation}\label{3.18}
\{ N(t)^{-1/2} u(t, N(t)^{-1} x + x(t))) : t \in I \}
\end{equation}
lies in a precompact subset of $L^{2}(\mathbb{R})$, and moreover, $N(t) \geq 1$ for all $t \in I$.

\begin{proposition}\label{t3.2}
If $u$ is an almost periodic solution to $(\ref{1.1})$ with $\| u \|_{L^{2}} < \| Q \|_{L^{2}}$ on a maximal interval $I \subset \mathbb{R}$ that satisfies $N(t) \geq 1$ for all $t \in I$, and $N(0) = 1$, then
\begin{equation}\label{3.19}
\int_{I} N(t)^{2} dt = \infty.
\end{equation}
\end{proposition}
\noindent \emph{Proof:} Again following \cite{dodson2017global}, suppose
\begin{equation}\label{3.20}
\int_{I} N(t)^{2} dt = R_{0} < \infty.
\end{equation}
Translating in space so that $x(0) = 0$, define the Morawetz potential
\begin{equation}\label{3.21}
M(t) = \int \psi(\frac{x}{R}) u(t,x)^{2} dx,
\end{equation}
where
\begin{equation}\label{3.22}
\psi(x) = \int_{0}^{x} \phi(t) dt,
\end{equation}
where $\phi$ is a smooth, even function, $\phi(x) = 1$ for $-1 \leq x \leq 1$, and $\phi$ is supported on $|x| \leq 2$.\medskip

Since $N(t) \geq 1$ and $\int_{I} N(t)^{2} dt < \infty$, $I$ is necessarily a finite interval. Therefore, $N(t) \nearrow +\infty$ as $t \rightarrow \sup(I)$ or $t \rightarrow \inf(I)$. Combining this with the fact that $|\dot{x}(t)| \lesssim N(t)^{2}$,
\begin{equation}\label{3.23}
\sup_{t \in I} |M(t)| \lesssim R_{0},
\end{equation}
with implicit constant independent of $R$. Moreover, by direct computation,
\begin{equation}\label{3.24}
\frac{d}{dt} M(t) = -3 \int \phi(\frac{x}{R}) u_{x}(t,x)^{2} dx + \frac{1}{R^{2}} \int \phi''(\frac{x}{R}) u(t,x)^{2} dx + \frac{5}{3} \int \phi(\frac{x}{R}) u(t,x)^{6} dx.
\end{equation}
Therefore, by the fundamental theorem of calculus,
\begin{equation}\label{3.25}
\int_{I} \int \phi(\frac{x}{R}) u_{x}(t,x)^{2} dx dt \lesssim R_{0} + \frac{|I|}{R^{2}} + \int_{I} \int u(t,x)^{6} dx dt.
\end{equation}
We have already demonstrated that the first two terms on the right hand side are uniformly bounded for any $R \geq 1$. So it remains to control the third term.\medskip

Partition $I$ into consecutive intervals
\begin{equation}\label{3.26}
I = \cup_{k} J_{k},
\end{equation}
where
\begin{equation}\label{3.27}
\int_{J_{k}} \int u(t,x)^{8} dx dt \sim 1.
\end{equation}
Using standard perturbation arguments, for any fixed $J_{k}$ with $t_{1}, t_{2} \in J_{k}$
\begin{equation}\label{3.28}
N(t_{1}) \sim N(t_{2}), \qquad \text{and} \qquad |t_{1} - t_{2}| \lesssim N(t_{1})^{-3}.
\end{equation}
Therefore, by H{\"o}lder's inequality,
\begin{equation}\label{3.29}
\int_{J_{k}} \int u(t,x)^{6} dx dt \lesssim |J_{k}|^{1/3} \| u \|_{L_{t}^{\infty} L_{x}^{2}}^{2/3} \| u \|_{L_{t,x}^{8}(J_{k} \times \mathbf{R})}^{16/3} \lesssim |J_{k}|^{1/3} \lesssim \int_{J_{k}} N(t)^{2} dt.
\end{equation}
Therefore,
\begin{equation}\label{3.30}
\int_{I} \int \phi(\frac{x}{R}) u_{x}(t,x)^{2} dx dt \lesssim R_{0} + \frac{|I|}{R^{2}}.
\end{equation}
Taking $R \rightarrow \infty$,
\begin{equation}\label{3.31}
\int_{I} \int u_{x}(t,x)^{2} dx dt \lesssim R_{0}.
\end{equation}
Therefore, by the Gagliardo-Nirenberg inequality, when $\| u \|_{L^{2}} < \| Q \|_{L^{2}}$, by conservation of energy,
\begin{equation}\label{3.32}
\int_{I} E(u(t)) dt = |I| E(u_{0}) \lesssim R_{0}.
\end{equation}
However, when $\| u_{0} \|_{L^{2}} < \| Q \|_{L^{2}}$, conservation of energy combined with $(\ref{3.32})$ contradicts the fact that $N(t)$ is unbounded on $I$, which completes the proof of Proposition $\ref{t3.2} $. $\Box$\medskip

Since the subsequence in the above analysis always converges strongly in $L^{2}$ to $u_{0}$, if we begin with an $ \delta $-close to $ Q $ solution, then the solution that we obtain is also $ \delta $-close to $ Q $.
This completes the proof of Theorem \ref{t3.3}.
\end{proof}

\section{Decomposition of the solution near a soliton}

Since after rescaling and translation, $u$ is close to $Q$, we can use a decomposition lemma of \cite{martel2002stability}. This lemma was proved when $u$ was close to $Q$ in $H^{1}$ norm, however, it is possible to prove a slightly weaker result when $u$ is merely close in $L^{2}$ norm.
\begin{lemma}\label{l2.2}
There exists $\delta > 0$ such that if
\begin{equation}\label{2.5}
\| u - \lambda_{0}(t)^{-1/2} Q(\frac{x - x_{0}(t)}{\lambda_{0}(t)}) \|_{L^{2}} < 2 \delta,
\end{equation}
then there exist $x(t)$ and $\lambda(t)$ such that
\begin{equation}\label{2.6}
\epsilon(t, y) := \lambda(t)^{1/2} u(t, \lambda(t) y + x(t)) - Q(y)
\end{equation}
satisfies
\begin{equation}\label{2.7}
(y Q_{y}, \epsilon) = (y(\frac{Q}{2} + y Q_{y}), \epsilon) = 0.
\end{equation}
Moreover,
\begin{equation}\label{2.8}
|\frac{\lambda_{0}(t)}{\lambda(t)} - 1| + |\frac{x_{0}(t) - x(t)}{\lambda(t)}| + \| \epsilon \|_{L^{2}} \lesssim \delta.
\end{equation}
\end{lemma}
\begin{remark}\label{r2.2}
Observe that by $(\ref{2.8})$, almost periodicity (according to Definition $\ref{d1.1}$) is maintained with the new $x(t)$ and $N(t) = \frac{1}{\lambda(t)}$.
\end{remark}
\noindent \emph{Proof:} Use the implicit function theorem. For  $\delta > 0$, let
\begin{equation}\label{2.9}
 U_{\delta}  = \{ u \in L^{2} : \| u - Q \|_{L^{2}} < 2 \delta \},
\end{equation}
and for $u \in L^{2}(\mathbb{R})$, $\lambda_{1} > 0$, $x_{1} \in \mathbb{R}$, define
\begin{equation}\label{2.10}
\epsilon_{\lambda_{1}, x_{1}}(y) = \lambda_{1}^{1/2} u(\lambda_{1} y + x_{1}) - Q.
\end{equation}
Define the functionals
\begin{equation}\label{2.11}
\rho_{\lambda_{1}, x_{1}}^{1}(u) = \int \epsilon_{\lambda_{1}, x_{1}} (y Q_{y}) dy, \qquad \rho_{\lambda_{1}, x_{1}}^{2}(u) = \int \epsilon_{\lambda_{1}, x_{1}} (y (\frac{Q}{2} + y Q_{y})) dy.
\end{equation}
Then by direct computation,
\begin{equation}\label{2.12}
\frac{\partial \epsilon_{\lambda_{1}, x_{1}}}{\partial x_{1}} = \lambda_{1}^{1/2} u_{x}(\lambda_{1} y + x_{1}),
\end{equation}
and
\begin{equation}\label{2.13}
\frac{\partial \epsilon_{\lambda_{1}, x_{1}}}{\partial \lambda_{1}} = \frac{1}{2} \lambda_{1}^{-1/2} u(\lambda_{1} y + x_{1}) + \lambda_{1}^{1/2} y u_{x}(\lambda_{1} y + x_{1}).
\end{equation}

Integrating by parts,
\begin{equation}\label{2.14}
\frac{\partial \rho_{\lambda_{1}, x_{1}}^{1}}{\partial x_{1}} = \int \lambda_{1}^{1/2} u_{x}(\lambda_{1} y + x_{1}) (y Q_{y})(y) dy = -\int \lambda_{1}^{-1/2} u(\lambda_{1} y + x_{1}) (y Q_{yy} + Q_{y})(y) dy,
\end{equation}
\begin{equation}\label{2.15}
\frac{\partial \rho_{\lambda_{1}, x_{1}}^{2}}{\partial x_{1}} = \int \lambda_{1}^{1/2} u_{x}(\lambda_{1} y + x_{1}) (\frac{y}{2} Q + y^{2} Q_{y})(y) dy = - \int \lambda_{1}^{-1/2} u(\lambda_{1} y + x_{1}) (\frac{Q}{2} + \frac{5y}{2} Q_{y} + y^{2} Q_{yy})(y) dy,
\end{equation}
\begin{equation}\label{2.16}
\aligned
\frac{\partial \rho_{\lambda_{1}, x_{1}}^{1}}{\partial \lambda_{1}} = \int [\frac{1}{2} \lambda_{1}^{-1/2} u(\lambda_{1} y + x_{1}) + \lambda_{1}^{1/2} y u_{x}(\lambda_{1} y + x_{1})] (y Q_{y})(y) dy \\ = \int \frac{1}{2} \lambda_{1}^{-1/2} u(\lambda_{1} y + x_{1}) y Q_{y}(y) dy -\int \lambda_{1}^{-1/2} u(\lambda_{1} y + x_{1}) y^{2} Q_{yy}(y) dy \\ -2 \int \lambda_{1}^{-1/2} u(\lambda_{1} y + x_{1}) y Q_{y}(y) dy,
\endaligned
\end{equation}
and
\begin{equation}\label{2.17}
\aligned
\frac{\partial \rho_{\lambda_{1}, x_{1}}^{2}}{\partial \lambda_{1}} = \int  [\frac{1}{2} \lambda_{1}^{-1/2} u(\lambda_{1} y + x_{1}) + \lambda_{1}^{1/2} y u_{x}(\lambda_{1} y + x_{1})] (\frac{y}{2} Q + y^{2} Q_{y})(y) dy \\ = \frac{1}{2} \int \lambda_{1}^{-1/2} u(\lambda_{1} x + x_{1}) (\frac{y}{2} Q + y^{2} Q_{y})(y) dy - \int \lambda_{1}^{-1/2} u(\lambda_{1} y + x_{1}) (y Q + \frac{7}{2} y Q_{y} + y^{3} Q_{y})(y) dy.
\endaligned
\end{equation}
This implies that $(\rho_{\lambda_{1}, x_{1}}^{1}, \rho_{\lambda_{1}, x_{1}}^{2})$ are $C^{1}$ functions of $(\lambda_{1}, x_{1})$.\medskip

Also,
\begin{equation}\label{2.18}
\aligned
\frac{\partial \rho_{\lambda_{1}, x_{1}}^{1}}{\partial x_{1}}|_{\lambda_{1} = 1, x_{1} = 0, u = Q} = \int Q_{y} \cdot y Q_{y} dy = 0, \\
\frac{\partial \rho_{\lambda_{1}, x_{1}}^{1}}{\partial \lambda_{1}}|_{\lambda_{1} = 1, x_{1} = 0, u = Q} = \int Q_{y} \cdot y(\frac{Q}{2} + y Q_{y}) = \int (\frac{Q}{2} + y Q_{y})^{2} dy > 0, \\
\frac{\partial \rho_{\lambda_{1}, x_{1}}^{2}}{\partial \lambda_{1}}|_{\lambda_{1} = 1, x_{1} = 0, u = Q} = \int (\frac{Q}{2} + y Q_{y}) y Q_{y} = \int (\frac{Q}{2} + y Q_{y})^{2} dy > 0, \\
\frac{\partial \rho_{\lambda_{1}, x_{1}}^{2}}{\partial x_{1}}|_{\lambda_{1} = 1, x_{1} = 0, u = Q} = \int  (\frac{Q}{2} + y Q_{y}) y(\frac{Q}{2} + y Q_{y}) dy = 0.
\endaligned
\end{equation}
Therefore, by the implicit function theorem, if
\begin{equation}\label{2.18.1}
\| u(x) - Q(x) \|_{L^{2}} < 2 \delta,
\end{equation}
then there exist $\lambda$, $x$ such that
\begin{equation}\label{2.19}
|\lambda - 1| + |x| + \| \epsilon \|_{L^{2}} \lesssim \| u - Q \|_{L^{2}} < 2 \delta,
\end{equation}
satisfying
\begin{equation}\label{2.20}
(\epsilon, y Q_{y}) = (\epsilon, y(\frac{Q}{2} + y Q_{y})) = 0.
\end{equation}

Now take a general $\lambda_{0}(t)$ and $x_{0}(t)$ such that
\begin{equation}\label{2.25}
\| u(y) - \lambda_{0}^{-1/2} Q(\frac{y - x_{0}}{\lambda_{0}}) \|_{L^{2}} < 2 \delta.
\end{equation}
Then after translation and rescaling,
\begin{equation}\label{2.26}
\| \lambda_{0}^{1/2} u(\lambda_{0} y + x_{0}) - Q(y) \|_{L^{2}} < 2 \delta.
\end{equation}
Then there exist $|\tilde{x}| + |\tilde{\lambda}| \lesssim \| u - Q \|_{L^{2}}$ such that
\begin{equation}\label{2.27}
\aligned
(\lambda_{0}^{1/2} (1 - \tilde{\lambda})^{1/2} u(\lambda_{0} (1 - \tilde{\lambda}) y + \lambda_{0} \tilde{x} + x_{0})&, x Q_{x}) = 0, \\
(\lambda_{0}^{1/2} (1 - \tilde{\lambda})^{1/2} u(\lambda_{0} (1 - \tilde{\lambda}) y + \lambda_{0} \tilde{x} + x_{0})&, x(\frac{Q}{2} + x Q_{x})) = 0.
\endaligned
\end{equation}
Since $|\tilde{\lambda}| \lesssim \delta$, $|\frac{\lambda_{0} (1 - \tilde{\lambda})}{\lambda_{0}}| \lesssim \delta$. Also, $|\lambda_{0} \tilde{x}| \lesssim \lambda_{0} \delta$, so $\frac{|x_{0}' - x_{0}|}{\lambda_{0}} \lesssim \delta$. This completes the proof of Lemma $\ref{l2.2}$. $\Box$\medskip

Introduce the variable
\begin{equation}\label{2.28}
s = \int_{0}^{t} \frac{dt'}{\lambda(t')^{3}}, \qquad \text{equivalently} \qquad \frac{ds}{dt} = \frac{1}{\lambda^{3}}.
\end{equation}

\begin{lemma}[Properties of the decomposition]\label{l2.3}
$(1)$ The function $\epsilon(s, y)$ satisfies the equation
\begin{equation}\label{2.30}
\aligned
\epsilon_{s} = (L \epsilon)_{y} + \frac{\lambda_{s}}{\lambda} (\frac{Q}{2} + y Q_{y}) + (\frac{x_{s}}{\lambda} - 1) Q_{y}
+ \frac{\lambda_{s}}{\lambda} (\frac{\epsilon}{2} + y \epsilon_{y}) + (\frac{x_{s}}{\lambda} - 1) \epsilon_{y} - (R(\epsilon))_{y},
\endaligned
\end{equation}
where
\begin{equation}\label{2.31}
L \epsilon = -\epsilon_{xx} + \epsilon - 5 Q^{4} \epsilon, \qquad \text{and} \qquad R(\epsilon) = 10 Q^{3} \epsilon^{2} + 10 Q^{2} \epsilon^{3} + 5 Q \epsilon^{4} + \epsilon^{5}.
\end{equation}

$(2)$ $\lambda$ and $x$ are $C^{1}$ functions of $s$ and
\begin{equation}\label{2.32}
\aligned
\frac{\lambda_{s}}{\lambda} (\int (\frac{Q}{2} + y Q_{y})^{2} dy - \int (2 y Q_{y} + y^{2} Q_{yy}) \epsilon dy) - (\frac{x_{s}}{\lambda} - 1) \int (y Q_{yy} + Q_{y}) \epsilon dy \\ = \int L(y Q_{yy} + Q_{y}) \cdot \epsilon dy - \int  R(\epsilon) (y Q_{y}) dy,
\endaligned
\end{equation}
and
\begin{equation}\label{2.33}
\aligned
-\frac{\lambda_{s}}{\lambda} \int \epsilon(y Q + \frac{7y^{2}}{2} Q_{y} + y^{3} Q_{yy}) dy + (\frac{x_{s}}{\lambda} - 1) (\int (\frac{Q}{2} + y Q_{y})^{2} dy - \int (\frac{Q}{2} + \frac{5y}{2} Q_{y} + y^{2} Q_{yy}) \epsilon dy) \\ = \int L(\frac{Q}{2} + \frac{5y}{2} Q_{y} + y^{2} Q_{yy}) \cdot \epsilon dy - \int (\frac{Q}{2} + \frac{5y}{2} Q_{y} + y^{2} Q_{yy}) R(\epsilon) dy.
\endaligned
\end{equation}
\end{lemma}
\emph{Proof:} See \cite{martel2002stability}. $\Box$\medskip

This lemma has an important corollary.
\begin{corollary}\label{c2.4}
For all $s \in \mathbb{R}$,
\begin{equation}\label{2.35}
|\frac{\lambda_{s}}{\lambda}| + |\frac{x_{s}}{\lambda} - 1| \lesssim \| \epsilon \|_{L^{2}} + \| \epsilon \|_{L^{2}} \| \epsilon \|_{L^{8}}^{4}.
\end{equation}
\end{corollary}
\begin{proof}
First observe that by H{\"o}lder's inequality and the boundedness of $Q$,
\begin{equation}\label{2.36}
\int L(y Q_{yy} + Q_{y}) \cdot \epsilon dy - \int  R(\epsilon) (y Q_{y}) dy \lesssim \| \epsilon \|_{L^{2}} + \| \epsilon \|_{L^{2}} \| \epsilon \|_{L^{8}}^{4},
\end{equation}
and
\begin{equation}\label{2.37}
\int L(\frac{Q}{2} + \frac{5y}{2} Q_{y} + y^{2} Q_{yy}) \cdot \epsilon dy - \int (\frac{Q}{2} + \frac{5y}{2} Q_{y} + y^{2} Q_{yy}) R(\epsilon) dy \lesssim \| \epsilon \|_{L^{2}} + \| \epsilon \|_{L^{2}} \| \epsilon \|_{L^{8}}^{4}.
\end{equation}
Since $\int (\frac{Q}{2} + y Q_{y})^{2} dy > 0$,
\begin{equation}\label{2.38}
\aligned
|\frac{\lambda_{s}}{\lambda}| (1 + O(\| \epsilon \|_{L^{2}})) + |\frac{x_{s}}{\lambda} - 1| O(\| \epsilon \|_{L^{2}}) \lesssim \| \epsilon \|_{L^{2}} + \| \epsilon \|_{L^{2}} \| \epsilon \|_{L^{8}}^{4}, \\
|\frac{\lambda_{s}}{\lambda}| O(\| \epsilon \|_{L^{2}}) + |\frac{x_{s}}{\lambda} - 1| (1 + O(\| \epsilon \|_{L^{2}})) \lesssim \| \epsilon \|_{L^{2}} + \| \epsilon \|_{L^{2}} \| \epsilon \|_{L^{8}}^{4},
\endaligned
\end{equation}
so after doing some algebra,
\begin{equation}\label{2.39}
|\frac{\lambda_{s}}{\lambda}| + |\frac{x_{s}}{\lambda} - 1| \lesssim \| \epsilon \|_{L^{2}} + \| \epsilon \|_{L^{2}} \| \epsilon \|_{L^{8}}^{4}.
\end{equation}
\end{proof}

Next, by Strichartz estimates, rescaling, and perturbation theory, for any $k \in \mathbb{Z}$,
\begin{equation}\label{2.40}
\| u(s, y) \|_{L_{s,x}^{8}([k, k + 1] \times \mathbb{R})} \lesssim \| u_{0} \|_{L^{2}} < \| Q \|_{L^{2}}.
\end{equation}
Therefore, by the triangle inequality,
\begin{equation}\label{2.42}
\| \epsilon \|_{L_{s,x}^{8}([k, k + 1] \times \mathbb{R})} \lesssim \| Q \|_{L^{8}} + \| u \|_{L_{s,x}^{8}} \lesssim 1.
\end{equation}

Also, by perturbative arguments, for $\| \epsilon_{0} \|_{L^{2}}$ sufficiently small, if $\lambda(k) = 1$ and $x(k) = 0$,
\begin{equation}\label{2.41}
\| u(t,x) - Q(x - t) \|_{L_{t}^{\infty} L_{x}^{2}([k, k + 1] \times \mathbf{R})} \lesssim \| \epsilon_{0} \|_{L^{2}}.
\end{equation}
Thus using scaling and translation symmetries, along with Strichartz estimates,
\begin{equation}\label{2.41.1}
\| \epsilon \|_{L_{s}^{\infty} L_{y}^{2}([k, k + 1] \times \mathbb{R})} + \| \epsilon \|_{L_{s,y}^{8}([k, k + 1] \times \mathbb{R})} \lesssim \| \epsilon(k) \|_{L^{2}}.
\end{equation}

Combining $(\ref{2.40})$, $(\ref{2.41.1})$, Lemma $\ref{l2.2}$, and the fact that $\| Q \|_{L^{8}}$ is uniformly bounded, along with choosing $\| \epsilon_{0} \|_{L^{2}}$ to be the infimum of $\| \epsilon \|_{L^{2}}$ on the interval $[k, k + 1]$,
\begin{equation}\label{2.43}
\int_{k}^{k + 1} |\frac{\lambda_{s}}{\lambda}|^{2} + |\frac{x_{s}}{\lambda} - 1|^{2} ds \lesssim \int_{k}^{k + 1} \| \epsilon \|_{L^{2}}^{2} ds + \| \epsilon \|_{L_{t}^{\infty} L_{x}^{2}([k, k + 1] \times \mathbb{R})}^{2} \int_{k}^{k + 1} \| \epsilon \|_{L^{8}}^{8} ds \lesssim \int_{k}^{k + 1} \| \epsilon \|_{L^{2}}^{2} ds.
\end{equation}

\section{Exponential decay estimates of $u$}

Having obtained a decomposition of $u$ close to the soliton, the next step is to prove exponential decay of a solution that stays close to $Q$ in the case when $N(t) \geq 1$ and $\int_{I} N(t)^{2} dt = \infty$. The proof follows a similar argument in \cite{merle2001existence} and utilizes the fact that $u$ is close to a soliton, and the soliton moves to the right while a dispersive solution moves to the left.\medskip

Recall that 
\begin{equation}\label{re1.8}
\sup_{t \in I} \|  \epsilon(t) \|_{L^{2}(\mathbb{R})} = 
\sup_{t \in I} \| u(t,x) - \frac{1}{\lambda(t)^{1/2}} Q(\frac{x - x(t)}{\lambda(t)}) \|_{L^{2}(\mathbb{R})} \lesssim \delta
\end{equation}

Observe that $N(t) \geq 1$ implies $\lambda(t) \lesssim 1$, where $\lambda(t) $ is given by Lemma \ref{l2.2}. It is convenient to rescale so that $\lambda(t) \leq 1$ for all $t \in I$. Note that after rescaling $N(t) \geq 1$. See Remark $\ref{r2.2}$.
\begin{lemma}[Exponential decay to the left of the soliton]\label{l4.1}
There exists some $a_{0}$ such that for $x_{0} \geq 10 a_{0}$, if $u$ satisfies Theorem $\ref{t3.3}$, $(\ref{1.4})$ and $ \| u_{0} \|_{L^{2}} < \| Q \|_{L^{2}}$, then 
\begin{equation}\label{4.1}
\| u(t, x + x(t)) \|_{L^{2}(x \leq -x_{0})}^{2} \leq 10 c_{1} e^{-\frac{x_{0}}{6}}.
\end{equation}
\end{lemma}
\noindent \textbf{Remark:} It is important to note that $a_{0}$ does not depend on the $\delta > 0$ in $(\ref{re1.8})$.\medskip

\begin{proof}
Suppose there exists some $t_{0} \in \mathbb{R}$ and $x_{0} \geq 10 a_{0}$ such that
\begin{equation}\label{4.2}
\int_{x \leq -x_{0}} u(t_{0}, x + x(t_{0}))^{2} dx > 10 c_{1} e^{-\frac{x_{0}}{6}}.
\end{equation}

Let $K = 3 \sqrt{2}$, and let
\begin{equation}\label{4.3}
\phi(x) = c Q(\frac{x}{K}),
\end{equation}
where
\begin{equation}\label{4.4}
c = \frac{1}{K \int_{-\infty}^{\infty} Q(x) dx}.
\end{equation}
Define
\begin{equation}\label{4.5}
\psi(x) = \int_{-\infty}^{x} \phi(y) dy.
\end{equation}
Then,
\begin{equation}\label{4.6}
\lim_{x \rightarrow -\infty} \psi(x) = 0, \qquad \lim_{x \rightarrow +\infty} \psi(x) = 1.
\end{equation}

Next, define a modification of $x(s)$, $\tilde{x}(s)$, such that $x(k) = \tilde{x}(k)$ for all $k \in \mathbb{Z}$, and for any $s \in \mathbb{R}$, and for any $k < s < k + 1$, $\tilde{x}(s)$ is the linear interpolation between $\tilde{x}(k)$ and $\tilde{x}(k + 1)$. Then by $(\ref{2.43})$,
\begin{equation}\label{4.14}
\aligned
\tilde{x}(k + 1) - \tilde{x}(k) = x(k + 1) - x(k) = \int_{k}^{k + 1} x_{s}(s) ds \\ = \int_{k}^{k + 1} \lambda(s) + (\sup_{k \leq s \leq k + 1} \lambda(s)) \cdot (\int_{k}^{k + 1} \| \epsilon \|_{L^{2}} ds)
= (\int_{k}^{k + 1} \lambda(s) ds) \cdot (1 + O(\delta)).
\endaligned
\end{equation}
The last estimate follows from the fact that $\lambda(s) \sim \lambda(k)$ for any $k \leq s \leq k + 1$. It also follows from $(\ref{4.14})$ that for any $k \leq s \leq k + 1$,
\begin{equation}\label{4.15}
|\tilde{x}(s) - x(s)| \leq |\tilde{x}(s) - \tilde{x}(k)| + |x(s) - x(k)| \lesssim (\int_{k}^{k + 1} \lambda(s) ds).
\end{equation}

For technical reasons, it is useful to consider two cases separately. First, suppose that
\begin{equation}\label{4.15.1}
\int_{0}^{\sup(I)} N(t)^{2} dt = \int_{\inf(I)}^{0} N(t)^{2} dt = +\infty.
\end{equation}
In this case, suppose without loss of generality that $t_{0} = 0$, where $t_{0}$ is given by $(\ref{4.2})$. Then,
\begin{equation}\label{4.2.1}
\int_{x \leq -x_{0}} u(0, x + x(0))^{2} dx > 10 c_{1} e^{-\frac{x_{0}}{6}}.
\end{equation}

Define the function
\begin{equation}\label{4.7}
I(t) = \int u(t,x)^{2} \psi(x - \tilde{x}(0) + x_{0} - \frac{1}{4} (\tilde{x}(t) - \tilde{x}(0))) dx.
\end{equation}
Then by $(\ref{4.2})$, since $\tilde{x}(0) = x(0)$,
\begin{equation}\label{4.8}
I(0) \leq \int u(0, x)^{2} dx - \frac{1}{2} \int_{x \leq -x_{0} + \tilde{x}(0)} u(0, x)^{2} dx \leq \int u(0,x)^{2} dx - 5 c_{1} e^{-\frac{x_{0}}{K}}.
\end{equation}
Integrating by parts,
\begin{equation}\label{4.17}
\aligned
I'(t) = -3 \int u_{x}(t,x)^{2} \phi(x - \tilde{x}(0) + x_{0} - \frac{1}{4} (\tilde{x}(t) - \tilde{x}(0))) dx \\
+ \int u(t,x)^{2} \phi''(x - \tilde{x}(0) + x_{0} - \frac{1}{4}(\tilde{x}(t) - \tilde{x}(0))) dx \\
+ \frac{5}{3} \int u(t,x)^{6} \phi(x - \tilde{x}(0) + x_{0} - \frac{1}{4}(\tilde{x}(t) - \tilde{x}(0))) dx \\
- \frac{\dot{\tilde{x}}(t)}{4} \int u(t,x)^{2} \phi(x - \tilde{x}(0) + x_{0} - \frac{1}{4}(\tilde{x}(t) - \tilde{x}(0))) dx.
\endaligned
\end{equation}
Following \cite{merle2001existence}, observe that
\begin{equation}\label{4.10}
\frac{\dot{\tilde{x}}(t)}{4} = \frac{\tilde{x}_{s}}{4 \lambda^{3}} = \frac{1}{4} \frac{1}{\lambda^{2}} \frac{\tilde{x}_{s}}{\lambda} = \frac{1}{4 \lambda^{2}} (1 + O(\delta)).
\end{equation}
Also observe that
\begin{equation}\label{4.11}
\phi''(x) = \frac{c}{K^{2}} Q_{xx}(\frac{x}{K}) \leq \frac{c}{K^{2}} Q(\frac{x}{K}) = \frac{1}{K^{2}} \phi(x) = \frac{1}{18} \phi(x).
\end{equation}
Since $\lambda(t) \leq 1$,
\begin{equation}\label{4.12}
-\frac{\dot{\tilde{x}}(t)}{4 \lambda^{2}} \phi(x) + \frac{1}{K^{2}} \phi(x) \leq -\frac{1}{18}.
\end{equation}
Therefore,
\begin{equation}\label{4.18}
\aligned
I'(t) \leq -3 \int u_{x}(t,x)^{2} \phi(x - \tilde{x}(0) + x_{0} - \frac{1}{4} (\tilde{x}(t) - \tilde{x}(0))) dx \\
+ \frac{5}{3} \int u(t,x)^{6} \phi(x - \tilde{x}(0) + x_{0} - \frac{1}{4}(\tilde{x}(t) - \tilde{x}(0))) dx \\
- \frac{1}{18} \int u(t,x)^{2} \phi(x - \tilde{x}(0) + x_{0} - \frac{1}{4}(\tilde{x}(t) - \tilde{x}(0))) dx.
\endaligned
\end{equation}

Next, using Lemma $6$ from \cite{merle2001existence} and H{\"o}lder's inequality,
\begin{equation}\label{4.19}
\aligned
\int_{|x - \tilde{x}(t)| > a_{0}} u(t,x)^{6} \phi(x - \tilde{x}(0) + x_{0} - \frac{1}{4}(\tilde{x}(t) - \tilde{x}(0))) dx \leq \| u^{2} \phi^{1/2} \|_{L^{\infty}(|x - \tilde{x}(t)| > a_{0}}^{2} (\int_{|x - \tilde{x}(t)| > a_{0}} u(t,x)^{2} dx) \\
\lesssim (\int_{|x - \tilde{x}(t)| > a_{0}} u(t,x)^{2} dx)^{2} (\int u_{x}(t,x)^{2} \phi(x - \tilde{x}(0) + x_{0} - \frac{1}{4} (\tilde{x}(t) - \tilde{x}(0))) dx \\ + \int u(t,x)^{2} \phi(x - \tilde{x}(0) + x_{0} - \frac{1}{4} (\tilde{x}(t) - \tilde{x}(0))) dx).
\endaligned
\end{equation}
Since $\lambda(t) \leq 1$ and $|x - \tilde{x}(t)| \lesssim 1$,
\begin{equation}\label{4.19.1}
\int_{|x - \tilde{x}(t)| > a_{0}} \lambda(t)^{-1} Q(\frac{x - x(t)}{\lambda(t)})^{2} dx \lesssim e^{-2 a_{0}},
\end{equation}
and by $(\ref{re1.8})$,
\begin{equation}\label{4.19.2}
\int_{|x - \tilde{x}(t)| > a_{0}} \lambda(t)^{-1} \epsilon(t, \frac{x - x(t)}{\lambda(t)})^{2} dx \leq \delta^{2}.
\end{equation}
Therefore, for $a_{0}$ sufficiently large, plugging $(\ref{4.19.2})$ into $(\ref{4.18})$,
\begin{equation}\label{4.20}
I'(t) \leq \int_{|x - \tilde{x}(t)| \leq a_{0}} u(t,x)^{6} \phi(x - \tilde{x}(0) + x_{0} - \frac{1}{4}(\tilde{x}(t) - \tilde{x}(0))) dx.
\end{equation}





By direct computation,
\begin{equation}\label{4.21}
\phi(x) \leq c e^{-\frac{1}{K} |x - \tilde{x}(0) + x_{0} - \frac{1}{4}(\tilde{x}(t) - \tilde{x}(0))|} = c e^{-\frac{1}{K} |x - \tilde{x}(t) + \frac{3}{4}(\tilde{x}(t) - \tilde{x}(0)) + x_{0}|}.
\end{equation}
Since $\tilde{x}(t) \geq \tilde{x}(0)$ and $|x - \tilde{x}(t)| \leq a_{0}$,
\begin{equation}\label{4.22}
= c e^{-\frac{1}{K}(x - \tilde{x}(t) + \frac{3}{4} (\tilde{x}(t) - \tilde{x}(0)) + x_{0})}.
\end{equation}
Therefore, since from $(\ref{4.10})$, $\dot{\tilde{x}}(t) \geq \frac{1}{2 \lambda^{2}}$, so
\begin{equation}\label{4.23}
I'(t) \leq C e^{\frac{-x_{0}}{K}} e^{-\frac{3}{4K}(\tilde{x}(t) - \tilde{x}(0))} \dot{\tilde{x}}(t) \int \lambda(t)^{2} u(t,x)^{6} dx.
\end{equation}
Making a change of variables, for any $T > 0$,
\begin{equation}\label{4.24}
\int_{0}^{T} I'(t) dt \leq \sum_{k \geq 0} C e^{\frac{-x_{0}}{K}} \int_{k}^{k + 1} x_{s}(s) e^{-\frac{3}{4K}(\tilde{x}(s) - \tilde{x}(0))} \int \lambda(s)^{2} u(t(s), x)^{6} dx ds.
\end{equation}
Then by $(\ref{2.40})$, conservation of mass, and a change of variables,
\begin{equation}\label{4.25}
(\ref{4.24}) \lesssim C K e^{\frac{-x_{0}}{K}}.
\end{equation}
However, by the fundamental theorem of calculus, $(\ref{4.8})$, the fact that by concentration compactness,
\begin{equation}\label{4.26.1}
I(t) \nearrow \int u(0,x)^{2} dx, \qquad \text{as} \qquad t \nearrow \sup(I),
\end{equation}
and $K = 3 \sqrt{2} > 6$ gives a contradiction for $a_{0}$ sufficiently large.\medskip
 
 Proving $(\ref{4.26.1})$ is the only place where $(\ref{4.15.1})$ is used. (Since $[0, t_{0}]$ is a compact set for any $t_{0} \in I$, and $N(t)$ is a continuous function, $(\ref{4.15.1})$ would also hold when $0$ is replaced by any $t_{0} \in I$.) Then by $(\ref{4.10})$, for any $T > 0$, $T \in I$,
 \begin{equation}\label{4.26.2}
 \tilde{x}(T) - \tilde{x}(0) = \int_{0}^{T} \dot{\tilde{x}}(t) dt \geq \int_{0}^{T} \frac{1}{2 \lambda^{2}} dt \sim \int_{0}^{T} N(t)^{2} dt \rightarrow +\infty,
 \end{equation}
 as $T \nearrow \sup(I)$. This proves $(\ref{4.26.1})$.
 \end{proof}

Now prove exponential decay to the right.
\begin{lemma}[Exponential decay to the right of the soliton]\label{l4.2}
For $x_{0} \geq 10 a_{0}$,
\begin{equation}\label{4.27}
\| u(t, x + x(t)) \|_{L^{2}(x \geq x_{0})}^{2} \leq 10 c_{1} e^{-\frac{x_{0}}{6}}.
\end{equation}
\end{lemma}
\begin{proof}
In this case, observe that if $u(t,x)$ solves $(\ref{1.1})$, then so does $v(t,x) = u(-t, -x)$. Once again assume without loss of generality that $(\ref{4.27})$ fails at $t_{0} = 0$. Define the function
\begin{equation}\label{4.28}
I(t) = \int v(t,x)^{2} \psi(x + \tilde{x}(0) + x_{0} + \frac{1}{4}(\tilde{x}(-t) - \tilde{x}(0))) dx.
\end{equation}
If $(\ref{4.27})$ fails at $t_{0} = 0$ for some $x_{0}$, then
\begin{equation}\label{4.28.1}
I(0) \leq \int u(t,x)^{2} dx - 5 c_{1} e^{-\frac{x_{0}}{6}}.
\end{equation}
Again by direct calculation,
\begin{equation}\label{4.29}
\aligned
I'(t) = -3 \int v_{x}(t,x)^{2} \phi(x + \tilde{x}(0) + x_{0} + \frac{1}{4}(\tilde{x}(-t) - \tilde{x}(0))) dx \\
+ \int v(t,x)^{2} \phi''(x + \tilde{x}(0) + x_{0} + \frac{1}{4}(\tilde{x}(-t) - \tilde{x}(0))) dx \\
+ \frac{5}{3} \int v(t,x)^{6} \phi(x + \tilde{x}(0) + x_{0} + \frac{1}{4}(\tilde{x}(-t) - \tilde{x}(0))) dx \\
- \frac{\dot{\tilde{x}}(-t)}{4} \int v(t,x)^{2} \phi(x + \tilde{x}(0) + x_{0} + \frac{1}{4}(\tilde{x}(-t) - \tilde{x}(0))) dx.
\endaligned
\end{equation}
Making the same argument as in Lemma $\ref{l4.1}$ and making a change of variables
\begin{equation}\label{4.30}
\aligned
I'(t) \leq \int_{|x - \tilde{x}(t)| \leq a_{0}} u(-t, -x)^{6} \phi(x + \tilde{x}(0) + x_{0} + \frac{1}{4}(\tilde{x}(-t) - \tilde{x}(0)) dx \\ = \int_{|x - \tilde{x}(t)| \leq a_{0}} u(-t, x)^{6} \phi(-x + \tilde{x}(0) + x_{0} + \frac{1}{4}(\tilde{x}(-t) - \tilde{x}(0)) dx.
\endaligned
\end{equation}
Then
\begin{equation}\label{4.31}
\phi(x) \leq C e^{-\frac{1}{K} (-x + \tilde{x}(-t) + \frac{3}{4} (\tilde{x}(0) - \tilde{x}(-t)) + x_{0})} \leq C e^{-\frac{3}{4K} (\tilde{x}(0) - \tilde{x}(-t)) - \frac{x_{0}}{K}}
\end{equation}
Therefore, as in Lemma $\ref{l4.1}$, we can show that
\begin{equation}\label{4.32}
\int_{0}^{T} I'(t) dt \lesssim CK e^{-\frac{x_{0}}{K}}.
\end{equation}
This proves $(\ref{4.27})$.
\end{proof}

\noindent \textbf{Remark:} Once again $K = 3 \sqrt{2}$.\medskip

It only remains to prove
\begin{theorem}\label{t4.3}
There does not exist an almost periodic solution to $(\ref{1.1})$ that satisfies $N(t) \geq 1$ for all $t \in I$,
\begin{equation}\label{4.33}
\int_{0}^{\sup(I)} N(t)^{2} dt = \infty,
\end{equation}
and
\begin{equation}\label{4.34}
\int_{\inf(I)}^{0} N(t)^{2} dt < \infty.
\end{equation}
\end{theorem}
\begin{proof}
By $(\ref{4.26.2})$ and $(\ref{4.33})$, exponential decay to the left must hold for such a solution. That is,
\begin{equation}\label{4.35}
\| u(t, x + x(t)) \|_{L^{2}(x \leq -x_{0})} \leq 10 c_{1} e^{-\frac{x_{0}}{K(u)}}.
\end{equation}
Now let $\chi$ be a smooth function such that $\chi(x) = 0$ for $x \leq 1$ and $\chi(x) = 1$ when $x > 2$. Then define the functional
\begin{equation}\label{4.36}
M(t) = \int \chi(\frac{x}{x_{0}}) u(t,x + x(0))^{2} dx.
\end{equation}
The fact that $N(t) \geq 1$ combined with $(\ref{4.34})$ implies $\inf(I) > -\infty$. This fact implies that $N(t) \nearrow \infty$ as $t \searrow \inf(I)$, so $(\ref{4.14})$ combined with almost periodicity imply that
\begin{equation}\label{4.37}
\lim_{t \searrow \inf(I)} M(t) = 0.
\end{equation}
Then integrating by parts,
\begin{equation}\label{4.38}
\aligned
\frac{d}{dt} M(t) = -\frac{3}{x_{0}} \int \chi'(\frac{x}{x_{0}}) u_{x}(t, x + x(0))^{2} + \frac{5}{3 x_{0}} \int \chi'(\frac{x}{x_{0}}) u(t, x + x(0))^{6} dx + \frac{1}{x_{0}^{3}} \int \chi'''(\frac{x}{x_{0}}) u(t, x + x(0))^{2} dx \\
\leq  \frac{5}{3 x_{0}} \int \chi'(\frac{x}{x_{0}}) u(t, x + x(0))^{6} dx + \frac{1}{x_{0}^{3}} \int \chi'''(\frac{x}{x_{0}}) u(t, x + x(0))^{2} dx.
\endaligned
\end{equation}
Then by $(\ref{2.40})$,
\begin{equation}\label{4.39}
\int_{\inf(I)}^{0} \frac{d}{dt} M(t) dt \lesssim \frac{5}{3 x_{0}} \int_{\inf(I)}^{0} N(t)^{2} dt - \frac{1}{x_{0}^{3}} \inf(I) \lesssim \frac{1}{x_{0}} (\int_{\inf(I)}^{0} N(t)^{2} dt).
\end{equation}
This implies that for any $t \in (\inf(I), 0]$,
\begin{equation}\label{4.40}
\int \chi(\frac{x}{x_{0}}) u(t, x + x(0))^{2} dx \lesssim \frac{1}{x_{0}} (\int_{\inf(I)}^{0} N(t)^{2} dt).
\end{equation}
Since $|\chi'(\frac{x}{x_{0}})| \leq \chi(\frac{2x}{x_{0}})$, plugging $(\ref{4.40})$ back in to $(\ref{4.39})$,
\begin{equation}\label{4.41}
\aligned
\int_{\inf(I)}^{0} \frac{d}{dt} M(t) dt \lesssim \| \chi(\frac{2 x}{x_{0}}) u \|_{L^{2}}^{2/3} \| u(t) \|_{L^{8}}^{16/3} dt \\ \lesssim \frac{1}{x_{0}^{1/3}}(\int_{\inf(I)}^{0} N(t)^{2} dt)^{1/3} \cdot \frac{1}{x_{0}} (\int_{\inf(I)}^{0} N(t)^{2} dt) = \frac{1}{x_{0}^{4/3}} (\int_{\inf(I)}^{0} N(t)^{2} dt)^{4/3}.
\endaligned
\end{equation}
Therefore, since $\int_{\inf(I)}^{0} N(t)^{2} dt = R < \infty$,
\begin{equation}\label{4.42}
\int_{x \geq 0} u(t, x + x(0))^{2} x dx < \infty,
\end{equation}
which combined with $(\ref{4.36})$ implies
\begin{equation}\label{4.43}
\int |x| u(t, x + x(0))^{2} dx < \infty.
\end{equation}
Then following the proof of Proposition $\ref{t3.2}$,
\begin{equation}\label{4.44}
\int_{\inf(I)}^{0} \int u_{x}(t,x)^{2} dx dt < \infty.
\end{equation}
By the Sobolev embedding theorem, $E(u) < \infty$. Then by conservation of energy and the Gagliardo-Nirenberg inequality, the solution to $(\ref{1.1})$ cannot blow up in finite time, which gives a contradiction.
\end{proof}

The proof that there does not exist a solution satisfying
\begin{equation}\label{4.45}
\int_{0}^{\sup(I)} N(t)^{2} dt < \infty, \qquad \int_{\inf(I)}^{0} N(t)^{2} dt = \infty,
\end{equation}
is identical.

\section{Virial identities}

Next, use the virial identity from \cite{martel2001instability} to show that, on average, the inner product $(\epsilon, Q)$ is bounded by $\| \epsilon \|_{L^{2}}^{2}$.
\begin{theorem}\label{t6.1}
For any $T > 0$,
\begin{equation}\label{6.1}
|\int_{0}^{T} \lambda(s)^{1/2} \int \epsilon(s,x) Q(x) dx ds| \lesssim C(u) + \int_{0}^{T} \lambda(s)^{1/2} \| \epsilon(s) \|_{L^{2}}^{2} ds.
\end{equation}
\end{theorem}
\begin{proof}
 Define the quantity,
\begin{equation}\label{6.2}
J(s) = \lambda(s)^{1/2} \int \epsilon(s, x) \int_{-\infty}^{x} (\frac{Q}{2} + z Q_{z}) dz dx - \lambda(s)^{1/2} \kappa,
\end{equation}
where $\kappa = \frac{1}{4} (\int Q)^{2}$. By rescaling, Lemmas $\ref{l4.1}$ and $\ref{l4.2}$, and the fact that $\lambda(s) \leq 1$,
\begin{equation}\label{6.3}
\sup_{s \in \mathbb{R}} J(s) < \infty.
\end{equation}
Then compute
\begin{equation}\label{6.4}
\frac{d}{ds} J(s) = \lambda(s)^{1/2} \int \epsilon_{s}(s, x) \int_{-\infty}^{x} (\frac{Q}{2} + z Q_{z}) dz dx + \frac{\lambda_{s}}{2 \lambda^{1/2}} \int \epsilon(s, x) \int_{-\infty}^{x} (\frac{Q}{2} + z Q_{z}) dz dx - \frac{\lambda_{s}}{2 \lambda^{1/2}} \kappa.
\end{equation}
Then taking the expression of $\epsilon_{s}$ given by $(\ref{2.30})$, and integrating by parts,
\begin{equation}\label{6.5}
-\int R(\epsilon)_{y} \int_{-\infty}^{y} (\frac{Q}{2} + z Q_{z}) dz dy = \int R(\epsilon) (\frac{Q}{2} + y Q_{y}) dy \lesssim \| \epsilon \|_{L^{2}}^{2} + \| \epsilon \|_{L^{2}} \| \epsilon \|_{L^{8}}^{4}.
\end{equation}
Next, integrating by parts, by $(\ref{2.35})$,
\begin{equation}\label{6.7}
(\frac{x_{s}}{\lambda} - 1) \int \epsilon_{y} \int_{-\infty}^{y} \frac{Q}{2} + z Q_{z} dz dy = -(\frac{x_{s}}{\lambda} - 1) \int \epsilon (\frac{Q}{2} + y Q_{y}) dy \lesssim \| \epsilon \|_{L^{2}}^{2} + \| \epsilon \|_{L^{2}}^{2} \| \epsilon \|_{L^{8}}^{4}.
\end{equation}
Next, integrating by parts and using $\epsilon \perp y(\frac{Q}{2} + y Q_{y})$,
\begin{equation}\label{6.8}
\aligned
\frac{\lambda_{s}}{\lambda} \int (\frac{\epsilon}{2} + y \epsilon_{y}) \int_{-\infty}^{y} (\frac{Q}{2} + z Q_{z}) dz dy = -\frac{1}{2} \frac{\lambda_{s}}{\lambda} \int \epsilon(s, y) \int_{-\infty}^{y} (\frac{Q}{2} + z Q_{z}) dz dy \\
- \frac{\lambda_{s}}{\lambda} \int \epsilon(s,y) y (\frac{Q}{2} + y Q_{y}) dy = -\frac{1}{2} \frac{\lambda_{s}}{\lambda} \int \epsilon(s, y) \int_{-\infty}^{y} (\frac{Q}{2} + z Q_{z}) dz dy.
\endaligned
\end{equation}
By direct calculation,
\begin{equation}\label{6.9}
(\frac{x_{s}}{\lambda} - 1) \int Q_{y} \int_{-\infty}^{y} \frac{Q}{2} + z Q_{z} dz dx = -(\frac{x_{s}}{\lambda} - 1) \int Q(\frac{Q}{2} + y Q_{y}) = 0.
\end{equation}
Also, since $Q$ is an even function,
\begin{equation}\label{6.10}
\frac{\lambda_{s}}{\lambda} \int (\frac{Q}{2} + y Q_{y}) \int_{-\infty}^{y} \frac{Q}{2} + z Q_{z} dz dx = \frac{\lambda_{s}}{\lambda} \frac{1}{2} (\int \frac{Q}{2} + y Q_{y} dy)^{2} = \frac{\lambda_{s}}{\lambda} \kappa.
\end{equation}
Finally, since $L$ is a self-adjoint operator,
\begin{equation}\label{6.11}
\int (L \epsilon)_{y} \int_{-\infty}^{y} \frac{Q}{2} + z Q_{z} dz = -\int (L \epsilon) (\frac{Q}{2} + y Q_{y}) dy = -\int \epsilon \cdot L(\frac{Q}{2} + y Q_{y}) dy.
\end{equation}

Now, by direct computation,
\begin{equation}\label{6.12}
\aligned
L(\frac{Q}{2} + x Q_{x}) = -\frac{Q_{xx}}{2} + \frac{Q}{2} - \frac{5}{2} Q^{5} - x Q_{xxx} - 2 Q_{xx} - 5 x Q^{4} Q_{x} + x Q_{x} \\ = x \partial_{x}(-Q_{xx} - Q^{5} + Q) - \frac{5}{2} (Q_{xx} + Q^{5}) + \frac{Q}{2} = -2 Q.
\endaligned
\end{equation}
Plugging this into $(\ref{6.11})$,
\begin{equation}\label{6.13}
(\ref{6.11}) = 2 \int Q \epsilon.
\end{equation}
Therefore, we have proved,
\begin{equation}\label{6.14}
\frac{d}{ds} J(s) = 2 \lambda(s)^{1/2} \int Q(y) \epsilon(s, y) dy + O(\lambda(s) \| \epsilon \|_{L^{2}}^{2}) + O(\lambda(s) \| \epsilon \|_{L^{2}} \| \epsilon \|_{L^{8}}^{4}).
\end{equation}
Using $(\ref{2.41})$ to estimate $\| \epsilon \|_{L_{s,y}^{8}}$ proves the theorem.
\end{proof}


We are now ready to finish the proof of the main result. 

\begin{proof}[{\bf Proof of Theorem $\ref{t1.1}$} ] 
Theorem $\ref{t1.1}$ may now be proved using a second virial identity. Let
\begin{equation}\label{5.1}
M(s) = \frac{1}{2} \lambda(s) \int y \epsilon(s, y)^{2} dy.
\end{equation}
Lemmas $\ref{l4.1}$ and $\ref{l4.2}$ imply that $(\ref{5.1})$ is uniformly bounded for all $s \in \mathbb{R}$.

Now, by the product rule,
\begin{equation}\label{2.48}
\frac{d}{ds} M(s) = \lambda(s) \int y \epsilon(s, y) \epsilon_{s}(s, y) dy + \frac{1}{2} \lambda_{s}(s) \int y \epsilon(s,y)^{2} dy.
\end{equation}
Again use $(\ref{2.30})$ to compute $\epsilon_{s}$. Integrating by parts,
\begin{equation}\label{2.49}
\aligned
\int y \epsilon (L \epsilon)_{y} dy = \int y \epsilon (-\epsilon_{yyy} + \epsilon_{y} - 20 Q^{3} Q_{y} \epsilon - 5 Q^{4} \epsilon_{y}) dy \\ = -\frac{3}{2} \int \epsilon_{y}^{2} dy - \frac{1}{2} \int \epsilon^{2} dy - 10 \int Q^{3} Q_{y} y \epsilon^{2} dy - \frac{5}{2} \int Q^{4} \epsilon^{2} dy =: H(\epsilon, \epsilon).
\endaligned
\end{equation}
Next, since $\epsilon \perp y Q_{y}$ and $\epsilon \perp y(\frac{Q}{2} + y Q_{y})$ for all $s \in \mathbb{R}$,
\begin{equation}\label{6.20}
\frac{\lambda_{s}}{\lambda} \int y \epsilon (\frac{Q}{2} + y Q_{y}) dy = (\frac{x_{s}}{\lambda} - 1) \int y \epsilon Q_{y} dy = 0.
\end{equation}
Next, integrating by parts and using $(\ref{2.35})$,
\begin{equation}\label{2.51}
(\frac{x_{s}}{\lambda} - 1) \int y \epsilon \epsilon_{y} = -(\frac{x_{s}}{\lambda} - 1) \int \epsilon^{2} dy \lesssim \| \epsilon \|_{L^{2}}^{3} (1 + \| \epsilon \|_{L^{8}}^{4}) \lesssim \| \epsilon \|_{L^{2}}^{3} + \| \epsilon \|_{L^{2}}^{11/2} \| \epsilon_{y} \|_{L^{2}}^{3/2}.
\end{equation}
Also,
\begin{equation}\label{2.52}
\aligned
-\int R(\epsilon)_{y} \epsilon(s,y) y dy = -\int y \epsilon (10 Q^{3} \epsilon^{2} + 10 Q^{2} \epsilon^{3} + 5 Q \epsilon^{4} + \epsilon^{5})_{y} dy = \frac{20}{3} \int Q^{3} \epsilon^{3} - 10 \int Q^{2} Q_{y} y \epsilon^{3} \\ -5 \int Q Q_{y} y \epsilon^{4} + \frac{15}{2} \int Q^{2} \epsilon^{4} + 4 \int Q \epsilon^{5} - \int Q_{y} y \epsilon^{5} +  \frac{5}{6} \epsilon^{6} \\ \lesssim \| \epsilon \|_{L^{2}}^{3/2} \| \epsilon \|_{L^{6}}^{3/2} + \| \epsilon \|_{L^{6}}^{6} \lesssim \| \epsilon \|_{L^{2}}^{5/2} \| \epsilon_{y} \|_{L^{2}}^{1/2} + \| \epsilon \|_{L^{2}}^{4} \| \epsilon_{y} \|_{L^{2}}^{2}.
\endaligned
\end{equation}
Finally, integrating by parts,
\begin{equation}\label{2.53}
 \frac{\lambda_{s}}{\lambda} \int y \epsilon (\frac{\epsilon}{2} + y \epsilon_{y}) =  -\frac{\lambda_{s}}{2 \lambda} \int y \epsilon^{2} = -\frac{\lambda_{s}}{ \lambda} M(s).
 \end{equation}
 Multiplying $(\ref{2.49})$--$(\ref{2.53})$ by $\lambda(s)$ and plugging in to $(\ref{2.48})$,
\begin{equation}\label{6.24}
\aligned
\int_{0}^{T} \lambda(s) H(\epsilon, \epsilon) ds \lesssim C(u) + \int_{0}^{T} \lambda(s) \| \epsilon \|_{L^{2}}^{3} + \lambda(s) \| \epsilon \|_{L^{2}}^{11/2} \| \epsilon_{y} \|_{L^{2}}^{3/2} ds \\ + \int_{0}^{T} \lambda(s) \| \epsilon \|_{L^{2}}^{5/2} \| \epsilon_{y} \|_{L^{2}}^{1/2} + \lambda(s) \| \epsilon \|_{L^{2}}^{4} \| \epsilon_{y} \|_{L^{2}}^{2} ds \lesssim C(u) + \delta \int_{0}^{T} \lambda(s) \| \epsilon \|_{L^{2}}^{2} ds + \delta \int_{0}^{T} \lambda(s) \| \epsilon_{y} \|_{L^{2}}^{2} ds.
\endaligned
\end{equation}
 The last inequality follows from $(\ref{re1.8})$.\medskip

Now then, take
\begin{equation}\label{6.25}
\epsilon_{1} =  \epsilon - \frac{(\epsilon, Q)}{\| Q \|_{L^{2}}^{2}} Q = \epsilon - a Q.
\end{equation}
Since $Q \perp x(\frac{Q}{2} + x Q_{x})$, $\epsilon_{1} \perp Q$ and $\epsilon_{1} \perp x(\frac{Q}{2} + x Q_{x})$. Therefore, from \cite{martel2000liouville}, there exists some $\delta_{1} > 0$ such that
\begin{equation}\label{6.26}
H(\epsilon_{1}, \epsilon_{1}) \geq \delta_{1} \| \epsilon_{1} \|_{H^{1}}^{2}.
\end{equation}
Also, integrating by parts,
\begin{equation}\label{6.27}
2 \lambda(s) H(\epsilon_{1}, aQ) + \lambda(s) H(aQ, aQ) \lesssim \lambda(s)^{1/2} |a| \cdot \lambda(s)^{1/2} \| \epsilon_{1} \|_{L^{2}} + \lambda(s) a^{2}.
\end{equation}
Therefore, $(\ref{6.24})$ and $(\ref{6.25})$ imply
\begin{equation}\label{6.27.1}
\aligned
\delta_{1} \int_{0}^{T} \lambda(s) \| \epsilon_{1} \|_{H^{1}}^{2} ds \lesssim C(u) + \delta \int_{0}^{T} \lambda(s) \| \epsilon \|_{H^{1}}^{2} ds + \int_{0}^{T} \lambda(s) a(s)^{2} ds + \int_{0}^{T} \lambda(s) a(s) \| \epsilon_{1} \|_{L^{2}} ds \\
 \lesssim C(u) + \delta \int_{0}^{T} \lambda(s) \| \epsilon_{1} \|_{H^{1}}^{2} ds + \int_{0}^{T} \lambda(s) a(s)^{2} ds + \int_{0}^{T} \lambda(s) a(s) \| \epsilon_{1} \|_{L^{2}} ds.
\endaligned
\end{equation}
Furthermore, for $\delta \ll \delta_{1}$, absorbing $ \delta \int_{0}^{T} \lambda(s) \| \epsilon_{1} \|_{H^{1}}^{2} ds$ into the left hand side,
\begin{equation}\label{6.27.2}
\aligned
\frac{\delta_{1}}{2} \int_{0}^{T} \lambda(s) \| \epsilon_{1} \|_{H^{1}}^{2} ds \lesssim C(u)  + \int_{0}^{T} \lambda(s) a(s)^{2} ds + \int_{0}^{T} \lambda(s) a(s) \| \epsilon_{1} \|_{L^{2}} ds.
\endaligned
\end{equation}
Also, by the Cauchy-Schwarz inequality,
\begin{equation}\label{6.28}
\frac{\delta_{1}}{4} \int_{0}^{T} \lambda(s) \| \epsilon_{1} \|_{H^{1}}^{2} ds \lesssim C(u) + \frac{1}{\delta_{1}} \int_{0}^{T} \lambda(s) (\epsilon, Q)^{2} ds.
\end{equation}
Also, since
\begin{equation}\label{6.29}
\| \epsilon \|_{H^{1}}^{2} \lesssim \| \epsilon_{1} \|_{H^{1}}^{2} + (\epsilon, Q) \| Q \|_{H^{1}}^{2},
\end{equation}
\begin{equation}\label{6.30}
\frac{\delta_{1}}{4} \int_{0}^{T} \lambda(s) \| \epsilon \|_{H^{1}}^{2} ds \lesssim C(u) + \frac{1}{\delta_{1}} \int_{0}^{T} \lambda(s) (\epsilon, Q)^{2} ds.
\end{equation}

Next, by conservation of mass and scaling invariance of the $L^{2}$ norm,
\begin{equation}\label{6.31}
\frac{1}{2} \| u_{0} \|_{L^{2}}^{2} = \frac{1}{2} \| Q + \epsilon \|_{L^{2}}^{2} = \frac{1}{2} \| Q \|_{L^{2}}^{2} + (\epsilon, Q) + \frac{1}{2} \| \epsilon \|_{L^{2}}^{2},
\end{equation}
and therefore, after doing some algebra,
\begin{equation}\label{6.32}
-(\epsilon, Q) = \frac{1}{2} \| Q \|_{L^{2}}^{2} - \frac{1}{2} \| u_{0} \|_{L^{2}}^{2} + \frac{1}{2} \| \epsilon \|_{L^{2}}^{2}.
\end{equation}
Since $\frac{1}{2} \| Q \|_{L^{2}}^{2} - \frac{1}{2} \| u_{0} \|_{L^{2}}^{2} > 0$ is a conserved quantity, it is convenient to label this quantity
\begin{equation}\label{6.33}
M = \frac{1}{2} \| Q \|_{L^{2}}^{2} - \frac{1}{2} \| u_{0} \|_{L^{2}}^{2}.
\end{equation}
Plugging $(\ref{6.32})$ into the right hand side of $(\ref{6.30})$,
\begin{equation}\label{6.34}
\frac{\delta_{1}}{4} \int_{0}^{T} \lambda(s) \| \epsilon \|_{H^{1}}^{2} ds \lesssim C(u) + \frac{M^{2}}{\delta_{1}} \int_{0}^{T} \lambda(s) ds + \frac{1}{\delta_{1}} \int_{0}^{T} \lambda(s) \| \epsilon \|_{L^{2}}^{4}.
\end{equation}
Since $\| \epsilon \|_{L^{2}} \lesssim \delta$, the second term in the right hand side may be absorbed into the left hand side, so
\begin{equation}\label{6.35}
\frac{\delta_{1}}{8} \int_{0}^{T} \lambda(s) \| \epsilon \|_{H^{1}}^{2} \lesssim C(u) + \frac{M^{2}}{\delta_{1}} \int_{0}^{T} \lambda(s) ds.
\end{equation}

Likewise, by Theorem $\ref{t6.1}$ and $(\ref{6.32})$,
\begin{equation}\label{6.36}
M \int_{0}^{T} \lambda(s)^{1/2} ds \lesssim C(u) + \int_{0}^{T} \lambda(s)^{1/2} \| \epsilon \|_{L^{2}}^{2} ds.
\end{equation}
Letting
\begin{equation}\label{6.37}
K = \int_{0}^{T} \lambda(s) ds, \qquad \text{and} \qquad R = \int_{0}^{T} \lambda(s)^{1/2} ds,
\end{equation}
combining $(\ref{6.35})$ and $(\ref{6.36})$,
\begin{equation}\label{6.38}
\frac{\delta_{1}}{8} \int_{0}^{T} \lambda(s) \| \epsilon \|_{H^{1}}^{2} ds \lesssim \frac{M K}{R \delta_{1}} \int_{0}^{T} \lambda(s)^{1/2} \| \epsilon \|_{L^{2}}^{2} ds + C(u) + \frac{MK}{R \delta_{1}} C(u).
\end{equation}

If it were the case that $\lambda(s) = 1$ for all $s \in \mathbb{R}$, (as in \cite{martel2001instability}), the proof would be complete, since in that case, $K = R = T$ and $M \lesssim \| \epsilon \|_{L^{2}} \leq \delta$, so for $\delta > 0$ sufficiently small, $(\ref{6.38})$ along with the fact that
\begin{equation}\label{6.39}
\lim_{T \nearrow \infty} \int_{0}^{T} \lambda(s) ds = \lim_{T \nearrow \infty} \int_{0}^{T} \lambda(s)^{1/2} ds = \infty,
\end{equation}
would imply that there exists a sequence $s_{n} \rightarrow +\infty$ such that
\begin{equation}\label{6.39.1}
\| \epsilon(s_{n}) \|_{H^{1}} \rightarrow 0,
\end{equation}
as $n \rightarrow \infty$. However, this would contradict the fact that $\| u_{0} \|_{L^{2}} < \| Q \|_{L^{2}}$.\medskip

In the general case, the proof will make use of the fact that $\lambda(s) \leq 1$ for all $s \in \mathbb{R}$ along with the fact that conservation of energy gives a lower bound (depending on $M$) on $\lambda(s)$.\medskip

Expanding out the energy,
\begin{equation}\label{6.40}
\aligned
E(Q + \epsilon) = \frac{1}{2} \int Q_{x}^{2} + \int Q_{x} \epsilon_{x} + \frac{1}{2} \int \epsilon_{x}^{2} \\ - \frac{1}{6} \int Q^{2} - \int Q^{5} \epsilon - \frac{5}{2} \int Q^{4} \epsilon^{2} - \frac{10}{3} \int Q^{3} \epsilon^{3} - \frac{5}{2} \int Q^{2} \epsilon^{4} - \int Q \epsilon^{5} - \frac{1}{6} \int \epsilon^{6}.
\endaligned
\end{equation}
First, note that
\begin{equation}\label{6.41}
E(Q) = \frac{1}{2} \int Q_{x}^{2} - \frac{1}{6} \int Q^{6} = 0.
\end{equation}
Next, integrating by parts, by $(\ref{6.32})$,
\begin{equation}\label{6.42}
\int Q_{x} \epsilon_{x} - \int Q^{5} \epsilon = -\int \epsilon(Q_{xx} + Q^{5}) = -\int \epsilon Q = M + \frac{1}{2} \int \epsilon^{2}.
\end{equation}
Therefore, by H{\"o}lder's inequality and the Sobolev embedding theorem,
\begin{equation}\label{6.43}
E(Q + \epsilon) = M + \frac{1}{2} \int \epsilon_{x}^{2} + \frac{1}{2} \int \epsilon^{2} - \frac{5}{2} \int Q^{4} \epsilon^{2} + O(\| \epsilon \|_{L^{2}}^{5/2} \| \epsilon \|_{H^{1}}^{1/2} + \| \epsilon \|_{L^{2}}^{4} \| \epsilon \|_{H^{1}}^{2}).
\end{equation}
Also, scaling symmetry implies
\begin{equation}\label{6.44}
E(Q + \epsilon) = \lambda(s)^{2} E_{0}.
\end{equation}

Recalling $(\ref{6.25})$ and $(\ref{6.26})$,
\begin{equation}\label{6.45}
\aligned
\frac{1}{2} \int \epsilon_{x}^{2} + \frac{1}{2} \int \epsilon^{2} - \frac{5}{2} \int Q^{4} \epsilon^{2} \geq \delta_{1} \| \epsilon_{1} \|_{H^{1}}^{2} - \frac{1}{\delta_{1}} (\epsilon, Q)^{2} \geq \delta_{1} \| \epsilon \|_{H^{1}}^{2} - \frac{2}{\delta_{1}} (\epsilon, Q)^{2} \\
\geq \delta_{1} \| \epsilon \|_{H^{1}}^{2} - \frac{2}{\delta_{1}} M^{2} - \frac{2}{\delta_{1}} \| \epsilon \|_{L^{2}}^{4} \geq \frac{\delta_{1}}{2} \| \epsilon \|_{H^{1}}^{2} - O(\frac{M^{2}}{\delta_{1}}).
\endaligned
\end{equation}
Since $M \lesssim \delta$ and $\| \epsilon \|_{L^{2}} \lesssim \delta$, for $\delta > 0$ sufficiently small,
\begin{equation}\label{6.46}
\lambda(s)^{2} E_{0} \geq \frac{\delta_{1}}{4} \| \epsilon \|_{H^{1}}^{2} + \frac{M}{2}.
\end{equation}

Since $E_{0}$ and both of the terms on the right hand side are positive, $(\ref{6.46})$ implies
\begin{equation}\label{6.47}
M \lesssim \lambda(s)^{2} E_{0},
\end{equation}
and therefore,
\begin{equation}\label{6.48}
\frac{M}{E_{0}} \lesssim \lambda(s)^{2}, \qquad \text{which implies} \qquad \lambda(s)^{-1/2} \lesssim (\frac{E_{0}}{M})^{1/4}.
\end{equation}

Plugging this into $(\ref{6.38})$,
\begin{equation}\label{6.49}
\frac{\delta_{1}}{8} \int_{0}^{T} \lambda(s) \| \epsilon \|_{H^{1}}^{2} ds \lesssim \frac{M^{3/4} E_{0}^{1/4} K}{R} \int_{0}^{T} \lambda(s) \| \epsilon \|_{L^{2}}^{2} ds + \frac{M K}{R \delta_{1}} C(u) + C(u).
\end{equation}
Since $\lambda(s) \leq 1$, $K \leq R$, so
\begin{equation}\label{6.50}
\frac{\delta_{1}}{8} \int_{0}^{T} \lambda(s) \| \epsilon \|_{H^{1}}^{2} ds \lesssim M^{3/4} E_{0}^{1/4} \int_{0}^{T} \lambda(s) \| \epsilon \|_{L^{2}}^{2} ds + C(u).
\end{equation}
Assuming for a moment that $E_{0} \lesssim 1$, $M \lesssim \delta$ and $(\ref{6.39})$ imply that $(\ref{6.39.1})$ must hold in this case as well, obtaining a contradiction.\medskip

The fact that $E_{0} \lesssim 1$ is a straightforward consequence of Lemmas $\ref{l4.1}$ and $\ref{l4.2}$. Suppose without loss of generality that
\begin{equation}\label{6.51}
\lambda(0) \geq \frac{1}{2} = \frac{1}{2} \sup_{s \in \mathbb{R}} \lambda(s).
\end{equation}
Lemmas $\ref{l4.1}$ and $\ref{l4.2}$ imply that
\begin{equation}\label{6.52}
\lambda(s) \int y \epsilon(s, y)^{2} dy \lesssim 1,
\end{equation}
with implicit constant independent of $u$, so long as $u$ satisfies $(\ref{re1.8})$. Then by $(\ref{6.30})$,
\begin{equation}\label{6.53}
\int_{0}^{1} \lambda(s) \| \epsilon \|_{H^{1}}^{2} ds \lesssim 1 + \frac{1}{\delta_{1}} \int_{0}^{1} \lambda(s) \| \epsilon \|_{L^{2}}^{2} ds.
\end{equation}
Since $(\ref{2.39})$ guarantees that $\lambda(s) \sim 1$ on $[0, 1]$,
\begin{equation}\label{6.54}
\int_{0}^{1} \| \epsilon \|_{H^{1}}^{2} ds \lesssim 1 + \frac{1}{\delta_{1}} \int_{0}^{1} \| \epsilon \|_{L^{2}}^{2} ds \lesssim 1.
\end{equation}
The last inequality follows from $(\ref{re1.8})$. Therefore, the proof that $E_{0} \lesssim 1$ is complete.
\end{proof}

\nocite{*}
\bibliography{biblio}
\bibliographystyle{abbrv}

\end{document}